\documentclass[]{article}

\usepackage{anyfontsize}
\usepackage{graphicx,psfrag,amssymb,amsmath,amsthm, bbm}
\usepackage[cmtip,all]{xy}
\usepackage{latexsym}
\usepackage{color}
\usepackage{hyperref, pdfsync, comment}
\usepackage{makeidx}
\usepackage{eucal}
\usepackage{nomencl}
\usepackage{cite}
\usepackage{authblk}

\newtheorem{thm}{Theorem}[section]
\newtheorem{theorem}{Theorem}

\newtheorem{lemma}[thm]{Lemma}

\newtheorem{proposition}[thm]{Proposition}

\newtheoremstyle{mydefinition}{}{}{\normalfont}{0pt}{\scshape}{.}{.5em}{}
\theoremstyle{mydefinition}
\newtheorem{defn}{Definition}
\newtheorem{definition}[defn]{Definition}

\newtheoremstyle{myremark}{}{}{\small\normalfont}{0pt}{\small\scshape}{.}{.5em}{}
\theoremstyle{myremark}
\newtheorem{remark}{Remark}

\numberwithin{equation}{section}

\def\beq{\begin{equation}}
\def\eeq{\end{equation}}
\def\beqn{\begin{equation*}}
\def\eeqn{\end{equation*}}
\def\beqy{\begin{eqnarray}}
\def\eeqy{\end{eqnarray}}
\def\beqyn{\begin{eqnarray*}}
\def\eeqyn{\end{eqnarray*}}

\def\dist{{\rm dist}}

\def\Cov{{\rm Cov}}

\def\bJ{\mathbf{J}}

\def\bc{\mathbf{c}}

\def\bn{\mathbf{n}}

\def\bp{\mathbf{p}}

\def\br{\mathbf{r}}
\def\bs{\mathbf{s}}

\def\bgamma{{\boldsymbol{\gamma}}}

\def\blambda{{\boldsymbol{\lambda}}}

\def\bdelta{{\boldsymbol{\delta}}}

\def\b1{{\boldsymbol{1}}}

\def\cA{\mathcal{A}}
\def\cB{\mathcal{B}}
\def\cC{\mathcal{C}}
\def\cD{\mathcal{D}}

\def\cG{\mathcal{G}}
\def\cH{\mathcal{H}}

\def\cO{\mathcal{O}}

\def\cS{\mathcal{S}}

\def\cV{\mathcal{V}}
\def\cW{\mathcal{W}}

\def\cZ{\mathcal{Z}}

\def\IE{{\mathbb E}}

\def\IN{{\mathbb N}}

\def\IR{{\mathbb R}}

\def\fB{\mathfrak{B}}

\def\fF{\mathfrak{F}}

\def\hM{\hat{M}}

\def\hcG{\hat{\mathcal{G}}}

\def\hmu{\widehat{\mu}}

\def\hnu{\widehat{\nu}}
\def\eps{\varepsilon}

\def\dist{\text{\rm dist}}

\def\eps{\varepsilon}

\def\hM{\widehat{M}}
\def\hT{\widehat{T}}
\def\hx{\widehat{x}}
\def\hD{\widehat{D}}

\def\oD{\overline{D}}

\def\brho{\overline{\rho}}
\def\bdelta{\overline{\delta}}
\def\bcA{\overline{\cA}}
\def\hcG{\widehat{\cG}}
\def\bcG{\overline{\cG}}
\def\wcG{\widetilde{\cG}}

\def\wcZ{\widetilde{\cZ}}
\def\blambda{\overline{\lambda}}
\def\wnu{\widetilde{\nu}}
\def\bbone{\mathbf{1}}
\def\fF{\mathfrak{F}}
\def\Var{\mathrm{Var}}
\def\oq{\overline{q}}
\def\hfB{\widehat{\fB}}
\def\om{\overline{m}}
\def\wm{\widehat{m}}
\def\hxi{\widehat{\xi}}

\def\homega{\widehat{\omega}}

\title{On Coupling Lemma and Stochastic Properties with Unbounded Observables for 1-d Expanding Maps}
\date{}

\author[1]{Jianyu Chen\thanks{E-mail: jychen@suda.edu.cn.}}
\author[2]{Hong-Kun Zhang\thanks{E-mail: hongkun@math.umass.edu.}}
\author[3]{Yiwei Zhang\thanks{E-mail: yiweizhang@hust.edu.cn. }}

\affil[1]{\footnotesize
School of Mathematical Sciences
\& Center for Dynamical Systems and Differential Equations\\
Soochow University\\
Suzhou, Jiangsu 215006, P.R.China}
\affil[2]{\footnotesize
Department of Mathematics and Statistics\\
University of Massachusetts Amherst\\
 Amherst, MA 01003, USA. }
 \affil[3]{\footnotesize
School of Mathematics and Statistics\\
Center for Mathematical Sciences\\
Hubei Key Laboratory of Engineering Modeling and Scientific Computing\\
Hua-Zhong University of Sciences and Technology\\
Wuhan 430074, China. }

\begin{document}

\maketitle

\begin{abstract}

%
%
%

In this paper, we establish a coupling lemma for standard families
in the setting of  piecewise expanding interval maps with countably many branches. Our method merely requires that the expanding map
satisfies Chernov's one-step expansion at $q$-scale and eventually covers a magnet interval.
Therefore, our approach is particularly powerful for maps whose inverse Jacobian has low regularity and
those who does not satisfy the big image property.
The main ingredients of our coupling method are two crucial lemmas:
the growth lemma in terms of the characteristic $\cZ$ function
and the covering ratio lemma over the magnet interval.
We first prove the existence of an absolutely continuous invariant measure.
What is more important, we further show that the growth lemma enables the liftablity of the Lebesgue measure to the associated Hofbauer tower, and the resulting invariant measure on the tower admits
a decomposition of Pesin-Sinai type.
Furthermore, we obtain
the exponential decay of correlations and the almost sure invariance principle (which is a functional version
of the central limit theorem).  For the first time, we are able to make a direct relation between
the mixing rates and the $\cZ$ function, see (\ref{equ:totalvariation1}).
The novelty of our results relies on establishing the regularity of invariant density,
as well as verifying the stochastic properties for a large class of unbounded observables.

Finally, we verify our assumptions for several well known examples that were previously studied
in the literature, and unify results to these examples in our framework.

\end{abstract}

\begin{center}


\end{center}

\tableofcontents

\section{Introduction}

The probabilistic coupling method is a flexible technique to compare
random processes with two different initial distributions. This method has been applied in a broad variety of contexts in modern probability theory, e.g., to prove limit theorems, to derive inequalities, or to obtain approximations.
For a comprehensive introduction on the developments of this topic,
we refer the readers to the books by Lindvall \cite{MR1180522} and Thorisson \cite{MR1741181}.
In the field of dynamical systems, the coupling method is also powerful and has been developed since  the celebrated work by Young \cite{MR1750438} for Young towers, and later
the systematic works by Chernov and Dolgopyat \cite{MR1782146, MR2219528, MR2499824} for introducing standard pairs for
chaotic billiards and partially hyperbolic systems. These two schemes have been adapted afterwards in various settings, e.g., \cite{MR1675304, MR1889567, MR2054836, MR2540156, Liv09, VZ16}.

In this paper, we are aiming to adapt the works in \cite{MR2540156, VZ16}, and to establish a widely applicable version of coupling lemma for standard pairs in the setting of piecewise expanding interval maps with countably many inverse branches. Roughly speaking, our coupling lemma indicates if the dynamical system
satisfies Chernov's one-step expansion condition
and eventually covers a magnet interval, then every two proper standard families can be coupled after iterations with an exponential decay for the tail of difference (see Theorem \ref{thm: equi}).
The assumptions of our coupling lemma are purely geometrical and simple to check (see Assumption \textbf{(H1)-(H3)} in Section \ref{sec: ass and results} for the precise statements). Moreover, these assumptions allow the systems under considerations to have lower regularity of the inverse Jacobian (see Assumption \textbf{(H2)})
and to merely satisfy a non-uniform version of ``big image property'', which is beyond a large part of the current theory of Markov maps with infinitely many branches (see Assumption \textbf{(H3)}).
Based on this coupling lemma, several statistical properties are further investigated, including the existence of an absolutely continuous invariant probability measure (see Theorem \ref{thm: acip}),
the regularity of the invariant density (see Theorem~\ref{thm: density}),
the exponential decay of correlations (see Theorem \ref{thm: mixing}) and the almost sure invariance principle (see Theorem \ref{thm: ASIP}), etc.



As the readers shall see later, our coupling lemma turns out to be completely independent of functional analysis for the transfer operator, and thus is effective on bypassing the difficulties on the construction of suitable Banach spaces.
Note that functional analytic method is extremely powerful on proving the spectral properties of
the transfer operators, and in particular, on establishing the regularity of
 the invariant density (See e.g. \cite{MR1679080, MR1938476,  MR2403704}).
We stress that
despite such analytic tool is in absence, we still manage to establish
the regularity of the invariant density by a completely new approach.
Another novelty that we would like to emphasize is that the observables
include a large class of unbounded functions, for which the exponential decay of correlations and the almost sure invariance principle hold.


The main ingredients in our proof are described as follows.
In order to obtain delicate estimates in our coupling algorithm,
we adopt the notion of characteristic $\cZ$ function,
which was first introduced in \cite{MR2540156}
(see also \S 7.4 in \cite{MR2229799} for an alternative form of the characteristic $\cZ$ function),
to measure the average length of standard families under the operations of cutting,
iterates and splitting over the magnet.
In particular, for the second operation, we establish the so-called growth
lemma with exponential rate (see Lemma~\ref{lem: growth}).
This key lemma is due to our assumption \textbf{(H1)}, and it guarantees that most of
intervals in standard families will grow after sufficient many iterates.
The other key lemma is the covering ratio lemma (see Lemma~\ref{lem: covering})
over a given magnet by standard families, which results from our assumption \textbf{(H3)}.
It follows that a fixed portion of standard families is coupled at times with bounded gap.

To the best of our knowledge, it is also the first time to introduce the crucial assumption \textbf{(H1)} - Chernov's one-step expansion at $q$-scale with the constant $q$ is allowed to be less than $1$ in the setting of interval dynamics.
The advantages of this assumption are two-fold. On the one hand, the coupling technique still works out, even though the inverse Jacobian of the expanding map may not be summable (see Section~\ref{sec: rem q-scale}). On the other hand, we are able to introduce the space $\cH_{\cW, \gamma, t}$ of
dynamically H\"older series (see Definition~\ref{def: series Holder}), which contains
a large class of unbounded functions when $t>0$.


We stress that the Hofbauer tower construction is used in the proof of Theorem~\ref{thm: density},
which shows that the invariant density is a dynamically H\"older series.
A important by-product is that we prove the Lebesgue measure is liftable to the Hofbauer tower, due to
the second growth lemma (see Lemma~\ref{second growth}).
Inspired by the work \cite{MR721733} of Pesin and Sinai,
we show that the limiting invariant measure on the Hofbauer tower
has a decomposition of Pesin-Sinai type,
and we further prove that the invariant measure on the unit interval is
in fact carried by a standard family.

In the last section of this paper, we apply our results in the following two aspects. Firstly, by revisiting several well known piecewise linear expanding maps in the literature, e.g., \cite{MR728198, MR2959300}, we provide a unified mechanism on the existence of absolutely continuous invariant probability measure, the regularity of the invariant density and some statistical properties for these examples (see Proposition \ref{prop:application}). Indeed, compared to Theorem 1 in  \cite{MR2959300}, our Assumption \textbf{(H1)} - Chernov's one-step expansion at $q$-scale turns out to be rather sharp on guaranteeing the existence of an absolutely continuous invariant probability measure. Secondly, we investigate the function space of dynamically H\"{o}lder series for which the almost sure invariant principle (ASIP) holds.
In particular, we are able to show the ASIP for the random process generated by certain unbounded observables over
the doubling map, which gives a functional improvement of the central limit theorem in the previous studies (See e.g.  \cite{MR2099550}).

The paper is organized as follows. In Section~\ref{sec: ass and results},
we introduce the general assumptions \textbf{(H1) - (H3)},
as well as the notions of standard pairs and standard families,
and then state the results on the coupling lemma and the consequent stochastic properties.
In Section~\ref{sec:prelim}, we make some preparations on
the quantitative behavior of the standard families under the dynamics. 
We then complete proofs of all the theorems in Section~\ref{sec: proofs1} -~\ref{Sec: Proof of ASIP}.
Finally, in Section~\ref{sec: rem app},
we provide some examples and remarks, for which our assumptions and results apply.

\section{Assumptions and Main Results}\label{sec: ass and results}

Let $M=[0,1]$ be the unit interval endowed with the standard Euclidean metric,
and let $m$ be the Lebesgue measure on $M$. Given a sub-interval $W\subset M$, we denote
its length by $|W|=m(W)$,
and the conditional measure of $m$ on $W$ by $m_W(\cdot)=m(\cdot\ | W)$.

We consider a one-dimensional map $T: M\circlearrowleft$ with countably many inverse branches,
that is, there is a countable partition $\xi_1$ of $M$ into sub-intervals,
on each interior of which $T$ is strictly monotonic and $C^1$-smooth.
Note that
we do not require $W\in \xi_1$ to be a maximal inverse branch.

\subsection{Assumptions.}\label{sec: assumptions}

In the following we list and briefly explain the assumptions. \\

Set $\xi_n=\xi_1\vee T^{-1} \xi_1\vee \dots \vee T^{-(n-1)} \xi_1$ for any $n\ge 1$.
Given an interval $W\subset M$, we let
$\{W_\alpha\}_{\alpha\in W/\xi_n}$ be the collection of sub-intervals of $W$ after being cut by $\cS_n$.
In other words, $\{W_\alpha\}_{\alpha\in W/\xi_n}$ is the relative partition of $W$ given by $\xi_n$.
For each $\alpha\in W/\xi_n$, we call the interval $T^n W_\alpha$
a \emph{component} of $T^nW$.
We further denote the collection of components of $T^nW$
by $\{T^nW_\alpha\}_{\alpha\in W/\xi_n}$.

Although some intervals in $\xi_n$ may be relatively short,
the following expansion condition ensures that
a large portion of  intervals in $\{T^nW_\alpha\}_{\alpha\in W/\xi_n}$ are relatively long.\\

\noindent {\textbf{(H1)}} Chernov's one-step expansion.
There exists $q\in (0,1]$ such that
\beq\label{def one-step}
\liminf_{\delta\to 0} \ \sup_{W\colon |W|<\delta} \ \sum_{{\alpha\in W/\xi_1}}
 \left(\frac{|W|}{|TW_{\alpha}|}\right)^{{q}}
 \frac{|W_{\alpha}|}{|W|}<1,
\eeq
where the supremum is taken over all sub-intervals $W\subset M$.

\begin{remark}
Assumption \textbf{(H1)} was brought up by Chernov and Zhang in \cite{MR2150341}
for chaotic billiards with polynomial mixing rates (with $q=1$),
and later in \cite{MR2540156, MR2903754, MR3123537, MR3168259} for
two-dimensional general hyperbolic systems with singularities.

To emphasize the choice of $q$, we shall call \eqref{def one-step} the (Chernov's) one-step expansion (condition) at $q$-scale.
Note that by Jensen's inequality, the one-step expansion at $q'$-scale implies
the one-step expansion at $q$-scale for any $0<q\le q'\le 1$.
In particular, the one-step expansion at $1$-scale, i.e.,
\beqn
\liminf_{\delta\to 0} \ \sup_{W\colon |W|<\delta} \ \sum_{{\alpha\in W/\xi_1}}
\frac{|W_\alpha|}{|TW_{\alpha}|}<1,
\eeqn
implies the one-step expansion at $q$-scale for any $q\in (0, 1)$.
In Section~\ref{sec: rem q-scale}, we shall provide a class of piecewise linear maps with infinitely many branches,
for which the one-step expansion fails at $1$-scale but holds at $q$-scale for some $q<1$.

Another advantage of Chernov's one-step expansion at $q$-scale with $q<1$
is that the observables that we consider can be unbounded (see Definition~\ref{def: series Holder}
and Remark~\ref{rem: observable q}).

\end{remark}

Let $\cS_n$ be the set of endpoints of intervals in $\xi_n$,
and set $\cS_\infty=\cup_{n\ge 1} \cS_n$.
It directly follows from Assumption \textbf{(H1)} that the map $T$ is uniformly expanding on
$M\backslash \cS_1$. Therefore, $\{\xi_1\}$ is a generating partition under iterations of $T^{-n}$, or equivalently,
$\xi_\infty:=\bigvee\limits_{k=0}^\infty T^{-k}\xi_1$
is the partition into individual points $\pmod m$,
which makes the \emph{separation time} given below well-defined on $M\backslash \cS_\infty$.

\begin{definition}\label{def: sep}
Given a pair of points $x$ and $y$ in $M\backslash \cS_\infty$,
the \emph{separation time} $\bs(x,y)$ is defined to be the smallest integer $n\ge 1$
such that $x$ and $y$ belong to distinct elements of $\xi_n$.
\end{definition}

To make assumptions on the regularity of Jacobian,
we first introduce the \emph{dynamically H\"older continuous functions}.

\begin{definition}
A function $f: M\to \IR$ is said to be \emph{dynamically H\"older continuous},
supported on an interval $W\subset M$ with parameter $\gamma\in (0, 1)$,
if $f|_{M\backslash W}\equiv 0$ and
\beqn
|f|_{W, \gamma}:=\sup\left\{\dfrac{|f(x)-f(y)|}{\gamma^{\bs(x,y)}}: \
x, y\in W\backslash \cS_\infty, \text{and}\ x\ne y \right\} <\infty.
\eeqn
We denote by $\cH_{W, \gamma}$ the space of such functions.
Note that $\cH_{W, \gamma}\subset L^\infty(m)$, and we denote
$\|f\|_{W, \gamma}:=\|f\|_{\infty} + |f|_{W, \gamma}$ for any $f\in \cH_{W, \gamma}$.
\end{definition}

Denote by $T'$ the derivative of $T$, which is well defined on $M\backslash \cS_1$.
We assume the following. \\

\noindent {\textbf{(H2)  Regularity of log Jacobian (with respect to $\xi_1$)}.
There exist $C_\bJ>0$ and $\gamma_\bJ\in (0, 1)$
such that for any interval $W\in \xi_1$,
the function $\bbone_W\cdot \log|T'|$ belongs to $\cH_{W, \gamma_\bJ}$
and $\left| \ \bbone_W\cdot \log|T'| \ \right|_{W, \gamma_\bJ}\le C_\bJ$. \\

Finally, since we do not have an invariant measure to begin with,
we impose the following topological condition in order to establish the coupling lemma.\\

\noindent {\textbf{(H3) Eventual covering}.}
There exists an interval $U$, which is called a \emph{magnet},
such that any interval $W\subset M$ will eventually covers $U$
in the following sense:
there is an integer $n_W\ge 1$ such that for any $n\ge n_W$,
at least one component of $T^nW$ contains $U$.

\begin{remark}
Our magnet interval is a topological analogy of the
magnet rectangle in two-dimensional hyperbolic systems,
see e.g. \S 7.12 in \cite{MR2229799}.

Assumption \textbf{(H3)} is easy to check when the map $T$
admits a Markov partition, of which $U$ is an element.
In general, this assumption may be verified by studying
the combinatorial structure of one-dimensional maps (see Section~\ref{sec: rem q-scale}).
\end{remark}

\subsection{Standard pairs and standard families}\label{Sec: SP}

To establish the coupling lemma for the one-dimensional maps, we introduce
the concepts of standard pairs and standard families.

Let $C_\bJ>0$ and $\gamma_\bJ\in (0,1)$ be constants given in Assumption \textbf{(H2)}.
Fix
\beq\label{regular constant}
\bgamma\in [\gamma_\bJ, 1), \ \ \text{and} \ \
C_{\br}\ge \max\{1, 2C_\bJ/(\bgamma^{-1}-1)\}.
\eeq

\begin{definition}[Pair and standard pair]\label{def: sp}
$(W, \nu)$ is called a \emph{pair} if
$W$ is an interval in $M$ and $\nu$ is an absolutely continuous probability measure supported on $W$.

A pair $(W,\nu)$ is called a \emph{standard pair} if
the density $\rho:=d\nu/dm$ is \emph{regular} on $W$ in the sense that
$\log \rho\in \cH_{W, \bgamma}$ with the semi-norm $|\log \rho|_{W, \bgamma}\le C_\br$.
\end{definition}

In the coupling process,
forward iterates of standard pairs require the definition of standard families,
which can be viewed as a convex sum of standard pairs.

\begin{definition}[Family and standard family]
Let $\cG=\{(W_{\alpha}, \nu_\alpha), \alpha\in \cA,  \lambda_\alpha\}$
be a countable family of pairs,
endowed with non-negative weights $\lambda_\alpha$ on the index set $\cA$.

The total measure of a family $\cG$ is given by
\beqn
\nu_\cG(A)=\sum_{\alpha\in \cA}\lambda_\alpha  \nu_{\alpha}(A),
\eeqn
for any Borel set $A\subset M$. For simplicity,
we also denote
\beqn
\cG=\sum_{\alpha\in \cA}\lambda_{\alpha}(W_{\alpha},\nu_{\alpha}) \ \ \text{and} \ \
\nu_\cG=\sum_{\alpha\in \cA}\lambda_{\alpha} \nu_{\alpha}.
\eeqn
A family $\cG$ is called a \emph{standard family} if each $(W_\alpha, \nu_\alpha)$ is a standard pair and
$\sum_{\alpha\in \cA} \lambda_\alpha=1$. \\

\end{definition}

We denote $\oq$ the supremum of scales for which Assumption \textbf{(H1)} holds, i.e.,
\beq\label{def oq}
\oq:=\sup\left\{ q\in (0, 1]:\ \text{the one-step expansion \eqref{def one-step} holds at}\  q\text{-scale} \right\}.
\eeq
From now on, we fix a scale $q_0\in (0, \oq)$.
Then there exists $\delta_0>0$ such that
\beq\label{def theta_0}
\theta_0:=\sup_{W\colon |W|<\delta_0} \ \sum_{{\alpha\in W/\xi_1}}
 \left(\frac{|W|}{|TW_{\alpha}|}\right)^{{q_0}}
 \frac{|W_{\alpha}|}{|W|}<1,
\eeq
where the supremum is taken over all sub-intervals $W\subset M$.

The average length of intervals in a family
$\cG=\sum_{\alpha\in \cA} \lambda_\alpha (W_{\alpha}, \nu_\alpha)$
is measured by the following characteristic $\cZ$ function
\beq\label{cZ}
\cZ(\cG):=\sum_{\alpha\in \cA} \lambda_\alpha |W_{\alpha}|^{-q_0}.
\eeq
Note that $\cZ(\cG)\ge 1$ for any standard family $\cG$.
Let $\fF$ be the collection of all families $\cG$ with $\cZ(\cG)<\infty$.
We fix constants
\beq\label{choose Zp}
c_0:=\max\left\{1,\ \frac{2\theta_0\delta_0^{-q_0}}{1-\theta_0}\right\}, \ \
\text{and} \ \
C_\bp\ge 10 c_0 e^{7C_\br}.
\eeq

\begin{definition}
A family $\cG$ is called \emph{proper} if $\cZ(\cG)\le C_\bp$.
\end{definition}

We say that two families $\cG_1$ and $\cG_2$ are equivalent
if $\nu_{\cG_1}=\nu_{\cG_2}$, denoted by $\cG_1\equiv \cG_2$.
Further, we denote $\cG\equiv \sum_{n=1}^\infty \cG_n$
if $\nu_\cG=\sum_{n=1}^\infty \nu_{\cG_n}$.

\subsection{Statement of results}

In this paper, we always assume that the map $T: M\circlearrowleft$ satisfies
Assumptions (\textbf{H1})-(\textbf{H3}) given in Section~\ref{sec: assumptions}.

\subsubsection{Coupling lemma}

With the preparations in Section~\ref{Sec: SP},
we are now ready to state our first main result - the coupling lemma over magnets.

\begin{theorem}\label{thm: coupling}
Given a magnet $U$, there exist $N_\bc\ge 1$ and $\Theta_\bc\in (0, 1)$ such that
the total measure of any proper standard family $\cG$ can be decomposed as
\beqn
\nu_\cG = \sum_{n=1}^\infty \nu_{n},
\eeqn
where each $\nu_n$ is a non-negative finite measure on $M$.
Moreover,
\begin{itemize}
\item[(1)] \textbf{Coupling}: If $n$ is an integer multiple of $N_\bc$, then
$T^{n}_*\nu_{n}=\Theta_\bc m_U$; otherwise, $\nu_n$ is null.
\item[(2)] \textbf{Exponential tail}:
$
\sum_{k> n}\nu_{k} (M)\le (1-\Theta_\bc)^{n/N_\bc}.
$
\end{itemize}
\end{theorem}

\begin{remark}
Note that the choices of $\{\nu_n\}_{n\ge 1}$ are not unique in the coupling lemma.
As we do not pursue the optimal values for the constants $N_\bc$ and $\Theta_\bc$,
we shall construct a slow coupling process in the proof of Theorem~\ref{thm: coupling}.
\end{remark}

\subsubsection{Absolutely continuous invariant measure}

The equidistribution property immediately follows from Theorem~\ref{thm: coupling}.

\begin{theorem}\label{thm: equi}
For any two proper standard families $\cG^1$ and $\cG^2$ and any $n\ge 0$,
\beqn
\left\| T^n_* \nu_{\cG^1} - T^n_*\nu_{\cG^2} \right\|_{TV} \le 2 (1-\Theta_\bc)^{n/N_\bc},
\eeqn
where $\|\cdot\|_{TV}$ denotes the total variation norm, and $\Theta_\bc, N_\bc$ are given by
Theorem~\ref{thm: coupling}.
\end{theorem}

The existence of an absolutely continuous invariant measure is a direct consequence of
Theorem~\ref{thm: equi}. Furthermore, iterates of any standard family converge exponentially to such measure.

\begin{theorem}\label{thm: acip}
There exists an absolutely continuous $T$-invariant probability measure $\mu$ on $M$.
Moreover, there exist constants $C_\bc>0$ and $\vartheta_\bc\in (0, 1)$, such that
for any standard family $\cG\in \fF$ and any $n\ge 0$,
\beq\label{equ:totalvariation1}
\left\| T^n_* \nu_{\cG} - \mu \right\|_{TV} \le C_\bc \vartheta_\bc^{n} \cZ(\cG).
\eeq
\end{theorem}
The above Theorem establish a new relationship between the $\cZ$ function and the rates of mixing for initial measures associated to standard families. Equation (\ref{equ:totalvariation1})  makes it a much clearer picture to understand that $\cZ$ function is the only factor that dominates the mixing rates for expanding maps.

In general, the invariant density $h=d\mu/dm\in L^{1}(m)$ could be unbounded when
$T$ has infinitely many inverse branches. To describe such function,
we introduce the space of
\emph{dynamically H\"older series}.

\begin{definition}\label{def: series Holder}
Let $\cW:=\{W_\alpha: \alpha\in \cA\}$ be a
collection of countably many intervals in $M$.
Choose $\gamma\in (0, 1)$ and $t\in [0, 1]$.
A function $f: M\to \IR$ is called a \emph{dynamically H\"older series}
supported on $\cW$ with parameter $\gamma$ and power $t$,
if $f=\sum_{\alpha\in \cA} f_\alpha$
such that each $f_\alpha\in \cH_{W_\alpha, \gamma}$ and
\beqn\label{Holder partition}
\|f\|_{\cW, \gamma, t}
:=\sum_{\alpha\in \cA} |W_\alpha|^{t} \|f_\alpha\|_{W_\alpha, \gamma}<\infty.
\eeqn
We denote by $\cH_{\cW, \gamma, t}$ the space of such functions.
\end{definition}

\begin{remark}\label{rem: observable q}
It is easy to see that $\cH_{\cW, \gamma, t}\subset \cH_{\cW, \gamma, t'}\subset L^1(m)$
for any $0\le t\le t'\le 1$, and $\cH_{\cW, \gamma, 0}\subset L^\infty(m)$.
In particular, if the collection $\cW=\{W\}$, then the space $\cH_{\cW, \gamma, 0}$
coincides with the space $\cH_{W, \gamma}$, which consists of dynamically H\"older continuous functions
supported on $W$ with parameter $\gamma$.
Also, the space $\cH_{\cW, \gamma, t}$ contains unbounded functions if
$t>0$ and the collection $\cW$ has intervals of arbitrary short length.
\end{remark}

Let $\bgamma$ be the constant given in \eqref{regular constant}, and
let $\oq$ be given in \eqref{def oq}.

\begin{theorem}\label{thm: density}
There exists a collection $\cW_h$ of countably many intervals such that
the invariant density $h=d\mu/dm \in \cH_{\cW_h, \bgamma, s}$ for any $s\in (1-\oq, 1]$.
\end{theorem}

\subsubsection{Stochastic properties}

In the rest of this subsection, we let $\bgamma$ and $\oq$ be given by \eqref{regular constant}
and \eqref{def oq}  respectively.
Also, let $\mu$ be the absolutely continuous invariant measure obtained in Theorem~\ref{thm: acip}.
We first show the system $(T, \mu)$ enjoys
exponential decay of correlations for dynamically H\"older series against bounded observables.

\begin{theorem}\label{thm: mixing}
For any  $t\in [0, \oq)$, there are constants $C_t>0$ and $\vartheta_t\in (0, 1)$ such that
for any $f\in \cH_{\cW, \bgamma, t}$ on some collection $\cW$ of countably many intervals
and for any $g\in L^\infty(m)$, we have
\beq\label{mu mixing}
\left| \int f  g\circ T^n   d\mu - \int f  d\mu \int g  d\mu \right|\le
C_t \vartheta_t^{n}  \|f\|_{\cW, \bgamma, t} \|g\|_{\infty}.
\eeq
\end{theorem}

Note that \eqref{mu mixing} is automatic for any bounded dynamically H\"older continuous
function $f\in \cH_{M, \bgamma}$.
In fact, for such bounded observables, we can show the exponential multiple decay of correlations,
and thus prove the central limit theorem (CLT) by the ``big small block technique" (see
\S 7.6-7.8 in \cite{MR2229799} for more details). Moreover, we can further establish
a functional generalization of the CLT -
the almost sure invariance principle (ASIP), which asserts that the stationary random process
$\{f\circ T^n\}_{n\ge 0}$ can be well approximated by a Brownian motion with an almost sure error.
We refer the readers to the papers \cite{MR2175992, MR2219528, MR2510014, MR2663640,
MR3152743, MR3646763, MR3784542} for the ASIP of stationary process generated
by bounded observables in various smooth dynamics.

However, when $f$ is an unbounded observable,
the CLT and ASIP may fail for
some obvious reasons, for instance, $f\not\in L^2(\mu)$ and thus
the corresponding process $\{f\circ T^n\}_{n\ge 0}$ has no finite variance.
In order to establish the limiting theorems for such process,
we need some to add some extra conditions, such as moment controls in \cite{MR3856951, MR3880492}.
In this paper, we impose the following conditions on the
dynamically H\"older series $f\in \cH_{\cW, \gamma, t}$.

\begin{definition}
Recall that $\cS_n$ is the set of endpoints of intervals in the partition $\xi_n$.
A collection $\cW=\{W_\alpha: \alpha\in \cA\}$ is \emph{adapted} if
for any $\alpha\in \cA$, there exists $n(\alpha)\in \IN$ such that
the two endpoints of $W_\alpha$ belong to $\cS_{n(\alpha)}$.

A function $f\in \cH_{\cW, \gamma, t}$ is \emph{adapted}
if the collection $\cW$ is adapted, and
\beq\label{Holder partition ad}
\|f\|^{ad}_{\cW, \gamma, t}
:=\sum_{\alpha\in \cA}  \mu(W_\alpha)^{t} \|f_\alpha\|_{W_\alpha, \gamma}<\infty.
\eeq
We denote by $\cH^{ad}_{\cW, \gamma, t}$ the space of functions satisfying \eqref{Holder partition ad}.

Assume that $t\in [0, \frac12)$. We further say that $f\in \cH^{ad}_{\cW, \gamma, t}$ has \emph{fast tail} if
there is $a>\max\left\{\frac{11}{2}, \ \frac{2+3t}{1-2t}\right\}$ such that
\beq\label{fast tail}
\sum_{\alpha\in \cA: \ n(\alpha)\ge n} \|f_\alpha\|_{L^{1/t}(\mu)}
=\cO\left( n^{-a} \right).
\eeq
\end{definition}

\begin{remark}
Note that $\cH^{ad}_{\cW, \gamma, t}\in L^{1/t}(\mu)\subset L^2(\mu)$ for $t\in [0, \frac12)$.
Also, it is automatic that
a dynamically H\"older function $f\in \cH_{W, \gamma}$ is
adapted and has fast tail if $W=M$ or $W\in \xi_n$ for some $n\ge 1$.
As we shall see in the proof of Theorem~\ref{thm: ASIP} below,
an adapted function $f\in \cH^{ad}_{\cW, \gamma, t}$
with fast tail can be well approximated by its conditional expectations with respect to the partition $\xi_n$.
\end{remark}

We denote $\IE(f)=\int f d\mu$ for any $f\in L^1(\mu)$,
and denote the covariance for $f, g\in L^2(\mu)$ by
$\Cov(f, g):=\IE(f g) -\IE(f) \IE(g)$. Then the variance of $f\in L^2(\mu)$
is  given by $\Var(f)=\Cov(f, f)$.

We now state the ASIP (and thus CLT) for the stationary process
generated by an adapted observable.

\begin{theorem}\label{thm: ASIP}
Fix any $t\in \left[0, \frac12 \right)$.
Let $f\in \cH^{ad}_{\cW, \bgamma, t}$ be of fast tail, such that
its auto-correlations satisfy that
\beq\label{cor 1m}
\left| \Cov(f, \ f\circ T^n) \right| =\cO\left( n^{-\frac{16}{15}} \right).
\eeq

%

Then the stationary process $\{f\circ T^n\}_{n\ge 0}$
satisfies the ASIP,
that is, there exist a constant $\lambda\in \left(0, \frac12\right)$ and a Wiener process $W(\cdot)$ such that
\beqn\label{ASIP12}
\left|\sum_{k=0}^{n-1} f \circ T^k - n \ \IE(f) - W\left(n \sigma_f^2 \right) \right|=\cO(n^{\lambda}), \  \ \text{a.s.}.
\eeqn
where $\sigma_f^2$ is given by the Green-Kubo formula, i.e.,
\beq\label{sigma_f}
\sigma^2_f:=\Var(f)+ 2\sum_{n=1}^\infty \Cov(f, \ f\circ T^n) \in [0, \infty).
\eeq
\end{theorem}

\begin{remark}
Condition \eqref{cor 1m} implies that the ASIP might hold for unbounded functions with
fairly  slow decay rates (in fact, polynomial decay) of auto-correlations.
We remark that the exponent $-\frac{16}{15}$ in \eqref{cor 1m}
is due to a classical result on invariance principle by Philipp and
Stout in \cite{MR0433597} (see Proposition~\ref{prop: ASIP} in Section~\ref{Sec: Proof of ASIP}).
Of course, we may improve this exponent by using some recent results on ASIP in probability theory,
but we shall not pursue it in this paper.
We shall provide an example in Section~\ref{sec: unbounded} on how to check Condition \eqref{cor 1m}.

By Theorem~\ref{thm: mixing}, Condition \eqref{cor 1m}
is automatic for any function $f\in \cH^{ad}_{\cW, \bgamma, 0}\subset L^\infty(m)$.
\end{remark}

\section{Quantitative Estimates on Standard Families}\label{sec:prelim}

In this section, we establish quantitative estimates on the density function and the average length (in terms of growth lemmas for the Characteristic $\cZ$ functions) for a standard family under iterates. These estimates will be the basis for understanding our coupling algorithm afterwards.


\subsection{Estimates for the density function on standard pairs}
We first provide the bounds for the density function of a standard pair, that is:


\begin{lemma}\label{lem: sp property}
If $(W, \nu)$ is a standard pair with the density function $\rho$, then
\beqn
e^{-C_\br}\le \frac{\rho(x)}{|W|^{-1} } \le  e^{C_\br},
\ \ \text{for any}\ x\in W.
\eeqn
Moreover, for any $x, y\in W$,
\beqn
\left|\rho(x)^{\pm 1} -\rho(y)^{\pm 1}\right|
\le \frac{C_\br e^{C_\br}}{|W|}  \bgamma^{\bs(x,y)}.
\eeqn
\end{lemma}

\begin{proof} By Definition \ref{def: sp} of a standard pair $(W,\nu)$, we have that for any $x, y\in W$,
\beqn
\rho(x) \le \rho(y) e^{C_\br \bgamma^{\bs(x,y)}} \le \rho(y) e^{C_\br}.
\eeqn
Taking integral over $W$ with respect to $dm(y)$ on both sides, we obtain
that $|W|\rho(x) \le e^{C_\br}$.
The proof for the other direction is similar.

Regarding the second assertion,
for any $z, w\in \IR$ with $|z|, |w|\le C_\br- \log |W|$,
\beqn
|e^z-e^w|\le |z-w| \sup_{|u|\le C_\br- \log |W|} |e^u| \le \frac{e^{C_\br}}{|W|}|z-w|,
\eeqn
and hence
{\allowdisplaybreaks
\beqyn
\left|\rho(x)^{\pm 1} -\rho(y)^{\pm 1}\right| = \left|e^{\pm \log \rho(x)} - e^{\pm \log \rho(y)}\right|
&\le & \frac{e^{C_\br}}{|W|}  \left|\log \rho(x) - \log \rho(y)\right| \\
&\le & \frac{C_\br e^{C_\br}}{|W|} \bgamma^{\bs(x,y)}.
\eeqyn
}This completes the proof of the lemma.
\end{proof}

The next lemma concerns the mergence of standard pairs over the same interval.

\begin{lemma}\label{lem: mergence}
Let $\{(W, \nu_\alpha)\}_{\alpha\in \cA}$ be
a countable collection of standard pairs.
For any non-negative weights $\lambda_\alpha$ on the index set $\cA$
such that $\sum_{\alpha\in \cA} \lambda_\alpha=1$,
then the mergence pair $(W, \nu)$ is also a standard pair,
where $\nu=\sum_{\alpha\in \cA} \lambda_\alpha \nu_\alpha$.
\end{lemma}

\begin{proof}
Let $\rho_\alpha$ be the density of $\nu_\alpha$, then
the density of mergence pair is given by $\rho=\sum_{\alpha\in \cA} \lambda_\alpha \rho_\alpha$.
By Definition of standard pairs, for any $x, y\in W$, we have
\beqn
e^{-C_\br \bgamma^{\bs(x,y)}}\le \frac{\rho_\alpha(x)}{\rho_\alpha(y)}\le  e^{C_\br \bgamma^{\bs(x,y)}}
\eeqn
and thus
\beqn
e^{-C_\br \bgamma^{\bs(x,y)}}\le \frac{\rho(x)}{\rho(y)}
=\frac{\sum_{\alpha\in \cA} \lambda_\alpha \rho_\alpha(x)}{\sum_{\alpha\in \cA} \lambda_\alpha \rho_\alpha(y)}
\le  e^{C_\br \bgamma^{\bs(x,y)}},
\eeqn
which immediately implies that $\left|\log \rho\right|_{W, \bgamma}\le C_\br$. So
the mergence pair $(W, \nu)$ is a standard pair.
\end{proof}

\subsection{Iterates of standard families}
\begin{definition}[Iterates of families]
For any integer $n\ge 0$ and any pair $(W, \nu)$,
let $\{W_\alpha\}_{\alpha\in W/\xi_n}$ be the relative partition of $W$ given by $\xi_n$,
and set $\nu_\alpha(\cdot):=\nu(T^{-n}(\cdot) |W_\alpha)$.
We define
\beqn
T^n(W, \nu)=\sum_{\alpha\in W/\xi_n} \nu(W_\alpha) \cdot (T^nW_\alpha, \nu_\alpha).
\eeqn
In general, for a family $\cG=\sum_{\beta\in \cA} \lambda_\beta   (W_{\beta}, \nu_\beta)$,
we define
\beqn
T^n \cG=\sum_{\beta\in \cA} \lambda_\beta \   T^n(W_{\beta}, \nu_\beta).
\eeqn
\end{definition}

\begin{lemma}\label{lem: forward iterates of sf}
If $\cG$ is a standard family, then $T^n\cG$ is also a standard family for any $n\ge 1$.
\end{lemma}

\begin{proof}
It suffices to show that for any standard pair $\cG=(W, \nu)$ with density $\rho=\frac{d\nu}{dm}$,
the first iterate
\beqn
T\cG=\sum_{\alpha\in W/\xi_1} \nu(W_\alpha) \cdot (TW_\alpha, \nu_\alpha)
\eeqn
is a standard family,
where $\{W_\alpha\}_{\alpha\in W/\xi_1}$ is the relative partition of $W$ given by $\xi_1$,
and $\nu_\alpha(\cdot)=\nu(T^{-1} (\cdot) |W_\alpha)$.
It is clear that
$
\sum_{\alpha\in W/\xi_1} \nu(W_\alpha)=\nu(W)=1,
$
and it remains to show that each $(TW_\alpha, \nu_\alpha)$ is a standard pair.
Indeed, for any Borel subset $A\subset TW_\alpha$,
\beqn
\frac{\nu_\alpha(A)}{m(A)}=\frac{\nu(T^{-1}A|W_\alpha)}{m(A)}=
\frac{1}{\nu(W_\alpha)} \frac{\nu(T^{-1}A\cap W_\alpha)}{m(T^{-1}A \cap W_\alpha)}
\frac{m(T^{-1}A\cap W_\alpha)}{m(A)}.
\eeqn
Since $T|_{W_\alpha}: W_\alpha\to TW_\alpha$ is invertible,
we denote $x_\alpha=(T|_{W_\alpha})^{-1}(x)$ for any $x\in TW_\alpha$.
Then the density function $\rho_\alpha:=\frac{d\nu_\alpha}{dm}$ is given by
\beq\label{density formula}
\rho_\alpha(x)=\frac{1}{\nu(W_\alpha)}\frac{\rho(x_\alpha)}{|T'(x_\alpha)|}.
\eeq

For any $x, y\in TW_\alpha$, by Assumption (\textbf{H2}) and
the choice of $\bgamma$ and $C_\br$ given by \eqref{regular constant},
{\allowdisplaybreaks
\beqyn
& & |\log \rho_\alpha(x) - \log \rho_\alpha(y)| \\
&\le & |\log \rho(x_\alpha) - \log \rho(y_\alpha)| + |\log |T'(x_\alpha)| - \log |T'(y_\alpha)|| \\
&\le & C_\br \bgamma^{\bs(x_\alpha, y_\alpha)} + C_\bJ \bgamma_{\bJ}^{\bs(x_\alpha, y_\alpha)}
\le (C_\br+C_\bJ)  \bgamma^{\bs(x, y)+1}  \le C_\br \bgamma^{\bs(x, y)}.
\eeqyn
}Hence the density $\rho_\alpha$ is regular on $TW_\alpha$.
This completes the proof of the lemma.
\end{proof}

\begin{remark}\label{remark non-standard 1}
Along the same lines in the proof of Lemma~\ref{lem: forward iterates of sf},
we can show that
$T\cG$ is a standard family if the family
$\cG=\sum_{\beta\in \cB} \lambda_\beta (W_{\beta}, \nu_\beta)$ is
a convex sum of pairs with densities $\rho_\beta=d\nu_\beta/dm$ satisfies that
\beq\label{non-standard 1}
 |\log \rho_\beta|_{W_\beta, \bgamma} \le \frac{1+\bgamma}{2\bgamma} C_\br.
\eeq

\end{remark}

\subsection{Cuttings of standard families}
\begin{definition}[Cut family]
Let $(W, \nu)$ be a pair, and $W$ is cut into
countable sub-intervals $\{W_i\}_{i\ge 1}$.
The cut family of $(W, \nu)$ is defined as
\beqn
(W, \nu)'= \sum_{i=1}^\infty
\nu(W_i)\cdot (W_i, \nu(\cdot|W_i)).
\eeqn
In general, let $\cG=\sum_{\alpha\in \cA} \lambda_\alpha (W_{\alpha}, \nu_\alpha)$ be a family.
Given an index subset $\cA'\subset \cA$ and a set $\cC$ of countable points in $M$,
we define the cut family $\cG'$ from $\cG$ with only pairs in $\cA'$ being cut by points in $\cC$,
that is,
\beqn
\cG'=\sum_{\alpha\in \cA} \lambda_\alpha (W_{\alpha}, \nu_\alpha)'.
\eeqn
We shall simply say $\cG'$ is a cut family from $\cG$ if there is no need to mention $\cA'$ and $\cC$. 
\end{definition}

It is easy to see that if $\cG$ is a standard family,
then any cut family $\cG'$ from $\cG$ is also a standard family.
Also, the cutting operation preserves the total measure,
while it does decrease the average length but not that much.
We recall that the average length of a family is represented by
the characteristic function $\cZ(\cdot)$ given by \eqref{cZ}.

\begin{lemma}\label{lem: cut length}
Let $\cG'$ be a cut family from a standard family $\cG$ by $k$ points,
\beqn
\cZ(\cG)\le \cZ(\cG')\le (k+1)e^{C_\br} \cZ(\cG).
\eeqn
\end{lemma}

\begin{proof} It suffices to show for a standard pair $\cG=(W, \nu)$, which is cut into
$(k+1)$ sub-intervals $W_1, W_2, \dots, W_{k+1}$. Then
\beqn
\cZ(\cG')=\sum_{1\le i\le k+1}\frac{\nu(W_i)}{|W_i|^{q_0}}
\ge \frac{\sum_{1\le i\le k+1} \nu(W_i)}{|W|^{q_0}}=\frac{1}{|W|^{q_0}}=\cZ(\cG).
\eeqn
On the other hand, by Lemma~\ref{lem: sp property},
\beqn
\nu(W_i)=\int_{W_i} \rho(x) dm(x)\le e^{C_\br} \frac{|W_i|}{|W|},
\eeqn
and thus,
{\allowdisplaybreaks
\begin{eqnarray*}
\cZ(\cG')=\sum_{1\le i\le k+1}\frac{\nu(W_i)}{|W_i|^{q_0}}
\le \sum_{1\le i\le k+1}  \frac{e^{C_\br} \frac{|W_i|}{|W|} }{|W_i|^{q_0}}
&=&\frac{e^{C_\br}}{|W|^{q_0}} \sum_{1\le i\le k+1} \left(\frac{|W_i|}{|W|}\right)^{1-q_0} \\
&\le & (k+1)e^{C_\br} \cZ(\cG).
\end{eqnarray*}
}This completes the proof of this lemma.
\end{proof}
\begin{remark}
It is not hard to check that if a family $\cG$ is a convex sum of countably many families,
say, $\cG=\sum_i \lambda_i \cG_i$, then
\beq\label{cZ decomposition}
\cZ(T^n \cG)=\sum_{i} \lambda_i \cZ(T^n \cG_i), \ \ \text{for any}\ n\ge 0.
\eeq
This together with Lemma~\ref{lem: cut length} implies that if $\cG'$ is a cut family from a standard family $\cG$, then
\beq\label{cZ cut}
\cZ(T^n \cG)\le \cZ(T^n \cG'), \ \ \text{for any}\ n\ge 0.
\eeq
\end{remark}

\subsection{Growth lemmas}


We establish the \emph{growth lemma} in this section. Roughly speaking, it means the value of $\cZ(T^n\cG)$ decreases exponentially in $n$ until it becomes small enough, providing that the initial standard family $\cG$ belongs to $\fF$, i.e.,
$\cZ(\cG)<\infty$. This fundamental property
was first introduced and proved by Chernov for dispersing billiards in \cite{MR1675363},
and later generalized by Chernov and Zhang in \cite{MR2540156}.


To begin with, we first state the growth lemma for the Lebesgue standard pairs.

\begin{lemma}\label{lem: growth Leb}
Let $\theta_0, \delta_0$ and $q_0$ be the constants given in \eqref{def theta_0}.
For any Lebesgue standard pair $(W, m_W)$ and any $n\ge 1$, we have
\beq\label{eq growth Leb}
\cZ(T^n (W, m_W) )\leq \theta_0^{n} \cZ((W, m_W))+2 \delta_0^{-q_0}(\theta_0+\dots+\theta_0^n).
\eeq
\end{lemma}

\begin{proof} For any $n\ge 0$,
we denote $\{W_\alpha\}_{\alpha\in W/\xi_n}$ the relative partition of $W$ given by $\xi_n$,
then
\beqn
T^n(W, m_W)=\sum_{\alpha\in W/\xi_n} m_W(W_\alpha)
\cdot (T^nW_\alpha,\ T^n_*m_{W_{\alpha}}),
\eeqn
and thus
{\allowdisplaybreaks
\beqy\label{cZ Tn Leb}
\cZ(T^n(W, m_W))
&=&\sum_{\alpha\in W/\xi_n} m_W(W_\alpha) |T^nW_\alpha|^{-q_0} \nonumber \\
&=&\sum_{\alpha\in W/\xi_n} \frac{|W_\alpha|}{|W|} \frac{1}{|T^nW_\alpha|^{q_0}}.
\eeqy
}

We now prove \eqref{eq growth Leb} by making induction on $n$.
When $n=1$, if $|W|<\delta_0$, then by \eqref{def theta_0},
\beq\label{Leb iterate}
\cZ(T(W, m_W))
=\sum_{\alpha\in W/\xi_1} \frac{|W_\alpha|}{|W|} \frac{1}{|TW_\alpha|^{q_0}}
\le \theta_0 |W|^{-q_0} = \theta_0 \cZ((W, m_W)).
\eeq
Otherwise, if $|W|\ge \delta_0$, we divide $(W, m_W)$ into $k=\lfloor |W|/\delta_0 \rfloor+1$
pieces $\{(W_1, m_{W_1}), \dots, (W_k, m_{W_k})\}$
of equal length which belongs to $[\delta_0/2, \delta_0)$.
In other words, $(W, m_W)$ is cut into a sum of standard pairs $\{(W_i, m_{W_i})\}_{1\le i\le k}$
with equal weights $1/k$.
By \eqref{cZ cut} and \eqref{Leb iterate}, we have
{\allowdisplaybreaks
\beqyn
\cZ(T(W, m_W))
\le \sum_{i=1}^k \frac{1}{k}  \cZ(T(W_i, m_{W_i}))
&\le &\frac{1}{k} \sum_{i=1}^k \theta_0 \cZ((W_i, m_{W_i})) \\
&\le &\theta_0 \left(\frac{\delta_0}{2}\right)^{-q_0}
\le 2\theta_0 \delta_0^{-q_0}.
\eeqyn
}In either case, we obtain \eqref{eq growth Leb} for $n=1$.

\medskip

Suppose now \eqref{eq growth Leb} holds for some $n$.
By \eqref{cZ Tn Leb},
{\allowdisplaybreaks
\beqyn
\cZ(T^{n+1}(W, m_W))
&=&
\sum_{\alpha\in W/\xi_1} \sum_{\beta\in W_\alpha/\xi_n}
\frac{|W_{\alpha \beta}|}{|W|} \frac{1}{|T^{n+1}W_{\alpha\beta}|^{q_0}} \\
&=&
\sum_{\alpha\in W/\xi_1}  \frac{|W_{\alpha}|}{|W|}
\sum_{\beta\in W_\alpha/\xi_n}  \frac{|W_{\alpha \beta}|}{|W_\alpha|} \frac{1}{|T^{n} (TW_{\alpha\beta})|^{q_0}} \\
&=&
\sum_{\alpha\in W/\xi_1}  \frac{|W_{\alpha}|}{|W|} \ \cZ(T^n(W_\alpha, m_{W_{\alpha}})) \\
&\le &
\sum_{\alpha\in W/\xi_1}  \frac{|W_{\alpha}|}{|W|}
\left[ \theta_0^n \cZ((W_\alpha, m_{W_\alpha})) + 2\delta_0^{-q_0} (\theta_0+\dots+\theta_0^n)
\right]  \\
&=& \theta_0^n \cZ(T(W, m_W)) + 2\delta_0^{-q_0} (\theta_0+\dots+\theta_0^n) \\
&\le & \theta_0^n (\theta_0 \cZ((W, m_W)) + 2\theta_0\delta_0^{-q_0})
+ 2\delta_0^{-q_0} (\theta_0+\dots+\theta_0^n) \\
&=& \theta_0^{n+1} \cZ((W, m_W))+2 \delta_0^{-q_0}(\theta_0+\dots+\theta_0^{n+1}).
\eeqyn
} Therefore, \eqref{eq growth Leb} also holds for $(n+1)$.
Thus, we complete the proof of this lemma by induction.
\end{proof}

Next, we present the growth lemma for all standard families in $\fF$.

\begin{lemma}\label{lem: growth} Let $c_0$ be given in \eqref{choose Zp}.
For any standard family $\cG\in \fF$ and any $n\ge 0$,
\beq\label{eq growth}
\cZ(T^n \cG)\leq e^{2 C_\br} \left(\cZ(\cG) \theta_0^{n} +c_0\right).
\eeq
\end{lemma}

\begin{proof}
By \eqref{cZ decomposition}, it is enough to prove \eqref{eq growth} for standard pairs.
Let $(W, \nu)$ be a standard pair with the density $\rho$.
For any $n\ge 0$,
we denote $\{W_\alpha\}_{\alpha\in W/\xi_n}$ the relative partition of $W$ given by $\xi_n$,
then the standard family $T^n(W, \nu)$ has weights $\nu(W_\alpha)$.
We consider the corresponding Lebesgue standard pair $(W, m_W)$, then
$T^n(W, m_W)$ has weights $m_W(W_\alpha)$. By Lemma \ref{lem: sp property},
\beq\label{Leb replace}
e^{-C_\br}\le \frac{\nu(W_\alpha)}{m_W(W_\alpha)}=
\frac{\int_{W_\alpha} \rho\ dm}{\int_{W_\alpha} |W|^{-1} dm}\le  e^{C_\br},
\eeq
which implies that
\beq\label{compare pairs}
e^{-C_\br}\le \frac{\cZ(T^n(W, \nu))}{\cZ(T^n(W, m_W))} \le e^{C_\br}.
\eeq
By Lemma \ref{lem: growth Leb} and the definition of $c_0$ in \eqref{choose Zp}, we have for any $n\ge 1$,
\beqn
\cZ(T^n(W, m_W)) \le \theta_0^n \cZ((W, m_W)) + c_0,
\eeqn
By \eqref{compare pairs},  \eqref{eq growth} holds for $\cG=(W, \nu)$.
\end{proof}

\begin{remark}\label{remark non-standard 2}
From the proofs of Lemma~\ref{lem: sp property} and Lemma~\ref{lem: growth},
we have that
\beqn
\cZ(T^n \cG)\leq e^{4 C_\br} \left(\cZ(\cG) \theta_0^{n} +c_0\right),
\eeqn
if the family $\cG=\sum_{\alpha\in \cA} \lambda_\alpha (W_{\alpha}, \nu_\alpha)\in \fF$ is a convex sum of pairs
with density $\rho_\alpha=d\nu_\alpha/dm$ satisfies that
\beq\label{non-standard 2}
|\log \rho_\alpha|_{W_\alpha, \bgamma} \le 2 C_\br.
\eeq
\end{remark}

\begin{lemma}\label{lem: growth proper}
For any standard family $\cG\in \fF$,
$T^n\cG$ is proper for any $n\ge n_\bp(\cG)$, where
\beqn 
n_\bp(\cG):=\left\lfloor\dfrac{- \log \cZ(\cG)}{\log \theta_0} \right\rfloor+1.
\eeqn
\end{lemma}
\begin{proof} By Lemma~\ref{lem: growth} and the definition of $C_\bp$ in \eqref{choose Zp}, for any $n\ge n_\bp(\cG)$,
\beqn
\cZ(T^n\cG)\le e^{2C_\br} (\cZ(\cG)\theta_0^n + c_0)\le e^{2C_\br}(1+c_0)<C_\bp,
\eeqn
and thus $T^n \cG$ is proper for any $n\ge n_\bp(\cG)$.
\end{proof}

We set
\beq\label{choose np}
n_\bp:=\left\lfloor\dfrac{- \log C_\bp}{\log \theta_0} \right\rfloor+1.
\eeq
If $\cG$ is a proper standard family, then $n_\bp(\cG)\le n_\bp$, and hence
$T^n\cG$ is proper for all $n\ge n_\bp$.

\section{Proof of Theorem~\ref{thm: coupling} -~\ref{thm: acip} }\label{sec: proofs1}

\subsection{Proof of Theorem \ref{thm: coupling}}

Throughout the section, let us fix a magnet $U$ given by Assumption (\textbf{H3}).
Theorem \ref{thm: coupling} will be proven by a coupling algorithm over $U$.
Before we describe the algorithm, let us first introduce two crucial lemmas (Lemma \ref{lem: split} and \ref{lem: covering})
whose proofs are postponed to Appendix~\ref{sec: appendix}.

\subsubsection{Lemmas for standard families over the magnet}

We first apply a special splitting of a standard family into
two parts, one of which is Lebesgue over the magnet $U$. To be more precise, let $\cG=\sum_{\alpha\in \cA} \lambda_\alpha (W_{\alpha}, \nu_\alpha)$ be a standard family.
The \emph{split family} from $\cG$ over the magnet $U$ with Lebesgue ratio $\brho\in (0, e^{-C_\br})$
is defined as
\beqn
\left(\brho \bdelta \right) \cdot \bcG
+ \left(1-\brho \bdelta \right)\cdot \hcG,
\eeqn
where
$
\bcA=\bcA(U):=\{\alpha\in \cA:\ {W_\alpha=U} \},
$
$\bdelta=\sum_{\alpha \in \bcA} \lambda_\alpha$,
the Lebesgue part $\bcG$ and the split part $\hcG$ are families \footnote{It is clear the splitting operation preserves the total measure, and
the Lebesgue part $\bcG$ is a standard family and $\bcG\equiv (U, m_U)$.
Although the split part $\hcG$ is a convex sum of pairs, it
might not be a standard family, since the pairs in the first summation of \eqref{G split 1}
may not have regular densities.
Also, the average length of $\hcG$ could become shorter than that of $\cG$.} given by
\beqn
\bcG:= \sum_{\alpha \in \bcA} \lambda_\alpha\bdelta^{-1} (W_\alpha, m_{W_\alpha})
=\sum_{\alpha \in \bcA} \lambda_\alpha\bdelta^{-1} (U, m_U)\equiv (U, m_U), \ \ \text{and} \ \
\eeqn
\beqy\label{G split 1}
\hcG:=
\sum_{\alpha \in \bcA} \dfrac{(1-\brho) \lambda_\alpha}{ 1-\brho \bdelta}
\left(W_\alpha, \dfrac{\nu_{\alpha}-\brho\ m_{W_\alpha}}{1-\brho} \right)
+ \sum_{\alpha\in \cA\backslash \bcA} \dfrac{  \lambda_\alpha}{1-\brho \bdelta }  (W_\alpha, \nu_{\alpha}).
\eeqy
With this convention, we have
\begin{lemma}\label{lem: split}
There is $\brho_\bc=\brho_\bc(U)\in (0, e^{-C_\br})$
such that for any $\brho \in (0, \brho_\bc)$ and standard family $\cG$,
we denote by $\hcG$ the split part of $\cG$ over the magnet $U$ with Lebesgue ratio $\brho$,
then $T\hcG$ is a standard family, and
\beqn
\cZ(T\hcG)\le e^{4C_\br} \left(\cZ(\cG) + c_0 \right).
\eeqn
\end{lemma}


\medskip

Next, We define the \emph{covering ratio} of a family
$\cG=\sum_{\alpha\in \cA} \lambda_\alpha  (W_{\alpha}, \nu_\alpha)$ over the magnet $U$ by
\beqn
\delta(\cG)=\sum_{\alpha\in \cA(U)} \lambda_\alpha,
\eeqn
where
$
\cA(U):=\{\alpha\in \cA:\ W_\alpha\ \text{contains} \ U\}.
$
Note that following properties of $\delta(\cdot)$ are straightforward from the definition.
\begin{itemize}
\item[(1)] If a family $\cG$ is a sum of countably many families, say, $\cG=\sum_i \lambda_i \cG_i$,
then for any $n\ge 0$,
\beq\label{delta p1}
\delta(T^n\cG)=\sum_{i} \lambda_i \delta(T^n\cG_i).
\eeq
\item[(2)]
By \eqref{Leb replace}, for any standard pair $(W, \nu)$ and any $n\ge 0$,
\beq\label{delta p2}
\delta(T^n(W, \nu))\ge e^{-C_\br} \delta(T^n(W, m_W)).
\eeq
\item[(3)]
If $\cG'$ is a cut family from a family $\cG$,
then for any $n\ge 0$.
\beq\label{delta p3}
\delta(T^n\cG)\ge \delta(T^n\cG').
\eeq
\item[(4)]
If $\cG$ is a standard family, and $\cG'$ is the cut family from $\cG$
with pairs in $\cA(U)$ being cut by the two endpoints of $U$, then
by Lemma~\ref{lem: sp property},
\beq\label{delta p4}
\delta(\cG')\ge e^{-C_\br} |U| \delta(\cG).
\eeq
\end{itemize}
Based on these properties, we have the following quantitative estimation on $\delta(\cdot)$.
\begin{lemma}\label{lem: covering}
There are $n_\bc=n_\bc(U)\ge n_\bp$ and $d_\bc=d_\bc(U)\in (0, 1)$ such that
for any proper standard family $\cG$, we have
$
\delta(T^{n_\bc}\cG)\ge d_\bc.
$
\end{lemma}

\subsubsection{The coupling algorithm for Theorem \ref{thm: coupling}}

We are now ready to describe our coupling algorithm. Fix a magnet U given by Assumption {\textbf (H3)}. Let $\brho_\bc\in (0, e^{-C_\br})$ be given by
Lemma~\ref{lem: split}, and
let $n_\bc\ge n_\bp$, $d_\bc\in (0, 1)$ be given by Lemma~\ref{lem: covering}.
Set
\beqn
\Theta_\bc:=e^{-C_\br} |U| d_\bc \brho_\bc.
\eeqn
Given a proper standard family $\cG$, we set $\hcG_0=\cG$ and $\wcG_0=T^{1+n_\bp} \hcG_0$.
By \eqref{choose np}, $\wcG_0$ is still a proper standard family.
Starting from $\wcG_0$, we apply the following inductive procedure. Assume that
a proper standard family $\wcG_k$ is defined, we shall obtain $\hcG_{k+1}$ and $\wcG_{k+1}$ as follows:
\begin{itemize}
\item[(1)] \textbf{Iteration:}
By Lemma~\ref{lem: covering}, $\delta(T^{n_\bc} \wcG_k)\ge d_\bc$.
Also, by \eqref{choose np}, $T^{n_\bc} \wcG_k$ is a proper standard family.
\item[(2)] \textbf{Cutting:}
Let $\cG_{k+1}'$ be the cut family from $T^{n_\bc}\wcG_k$ with pairs that contains $U$
being cut by the two endpoints of $U$. By Inequality \eqref{delta p4},
$\delta(\cG_{k+1}')\ge e^{-C_\br} |U| d_\bc = \Theta_\bc /\brho_\bc$.
By Lemma~\ref{lem: cut length},
$\cZ(\cG_{k+1}')\le 3e^{C_\br} \cZ(T^{n_c} \wcG_k)\le 3 e^{C_\br} C_\bp$.
\item[(3)] \textbf{Splitting:}
Set $\brho_{k+1}:= \Theta_\bc / \delta(\cG_{k+1}')$. We split $\cG_{k+1}'$
over the magnet $U$ with Lebesgue ratio $\brho_{k+1}$, and obtain
\beq\label{G split 2}
\cG_{k+1}'\equiv \Theta_\bc \bcG_{k+1} + (1-\Theta_\bc) \hcG_{k+1},
\eeq
where $\bcG_{k+1}$ is the Lebesgue part and $\hcG_{k+1}$ is the split part.
By Lemma~\ref{lem: split}, $T\hcG_{k+1}$ is a standard family, and
\beqn
\cZ(T\hcG_{k+1})\le e^{4C_\br} (\cZ(\cG_{k+1}') + c_0)\le 3 e^{5C_\br} (C_\bp + c_0).
\eeqn
By Lemma~\ref{lem: growth}, Equations \eqref{choose Zp} and \eqref{choose np},
{\allowdisplaybreaks
\beqyn
\cZ(T^{1+n_\bp} \hcG_{k+1})
&\le & e^{2C_\br}(\cZ(T\hcG_{k+1}) \theta_0^{n_\bp} +c_0) \\
&\le & e^{2C_\br}( 3 e^{5C_\br} (C_\bp + c_0) \theta_0^{n_\bp} +c_0) \\
&\le & e^{2C_\br}( 3 e^{5C_\br} (1 + c_0)  +c_0) \\
&\le & 7c_0 e^{7C_\br} < C_\bp.
\eeqyn
}Therefore, $\wcG_{k+1}:=T^{1+n_\bp} \hcG_{k+1}$ is a proper standard family.
\end{itemize}

Set $N_\bc:=(1+n_\bp+n_\bc)$. At the $k$-th step of the above coupling construction,
the Lebesgue part $\bcG_k=\sum_{\alpha\in \bcA_k} \blambda_\alpha (U, m_U)$ has the following property:
the index set $\bcA_k\subset M/\xi_{k N_\bc}$, and there is an interval $W_\alpha$ inside some element of $\xi_{k N_\bc}$
such that $T^{k N_\bc} W_{\alpha}=U$. In particular, $T^{k N_\bc}$ is invertible on $W_\alpha$.
Then we can define the family
\beqn
T^{-kN_\bc} \bcG_{k}
:=\sum_{\alpha} \blambda_\alpha \left(W_\alpha, \left[(T^{kN_\bc}|_{W_\alpha})^{-1}\right]_* m_U\right).
\eeqn
For any $n\ge 1$, we define
\beqn
\cG_n:=\begin{cases}
\Theta_\bc (1-\Theta_\bc)^{k-1} T^{-kN_\bc} \bcG_{k}, \ & \ \text{if}\ n=kN_\bc, \\
\text{null}, \ & \ \text{otherwise}.
\end{cases}
\eeqn
It is easy to see that Statement (1) of Theorem \ref{thm: coupling} holds with $\nu_n:=\nu_{\cG_n}$.  
By \eqref{G split 2}, we have for any $k\ge 1$,
{\allowdisplaybreaks
\beqyn
T^{kN_\bc}\cG=T^{kN_\bc} \hcG_0
&\equiv & T^{(k-1)N_\bc} \left(\Theta_\bc \bcG_{1} + (1-\Theta_\bc) \hcG_{1}\right) \\
&\equiv & \Theta_\bc T^{(k-1)N_\bc} \bcG_{1} + (1-\Theta_\bc) T^{(k-1)N_\bc} \hcG_{1} \\
&\equiv & T^{kN_\bc} \cG_{N_\bc} + (1-\Theta_\bc) T^{(k-2)N_\bc}
\left(\Theta_\bc \bcG_{2} + (1-\Theta_\bc) \hcG_{2}\right) \\
&\equiv & T^{kN_\bc} \cG_{N_\bc}  + \Theta_\bc (1-\Theta_\bc) T^{(k-2)N_\bc} \bcG_{2}
+(1-\Theta_\bc)^2 T^{(k-2)N_\bc} \hcG_2 \\
&\equiv & T^{kN_\bc} \cG_{N_\bc} + T^{kN_\bc} \cG_{2N_\bc} +(1-\Theta_\bc)^2 T^{(k-2)N_\bc} \hcG_2 \\
&\equiv & \dots \\
&\equiv & T^{kN_\bc} \left( \sum_{i=1}^k \cG_{i N_\bc} \right) + (1-\Theta_\bc)^k \hcG_k.
\eeqyn
}It is obvious that $\nu_\cG - \sum_{i=1}^k \nu_{\cG_{i N_\bc}}$ is a non-negative measure, and thus,
{\allowdisplaybreaks
\beqy\label{exp tail est}
\left\|\nu_\cG - \sum_{i=1}^k \nu_{\cG_{i N_\bc}}\right\|_{TV}
&=& \left(\nu_\cG - \sum_{i=1}^k \nu_{\cG_{i N_\bc}}\right)(M) \nonumber \\
&=& T^{kN_\bc}_*\left(\nu_\cG - \sum_{i=1}^k \nu_{\cG_{i N_\bc}}\right)(T^{kN_\bc} M) \nonumber \\
&\le &(1-\Theta_\bc)^k,
\eeqy
}which implies that
\beqn
\nu_\cG\equiv \sum_{k=1}^\infty \nu_{\cG_{kN_{\bc}}} =\sum_{n=1}^\infty \nu_n.
\eeqn
This provides the decomposition of $\nu_\cG$ in Theorem \ref{thm: coupling}.
Moreover, the exponential tail bound in Statement (2) directly follows from \eqref{exp tail est}.
Therefore, the proof of Theorem \ref{thm: coupling} is complete.

\subsection{Proof of Theorem~\ref{thm: equi}}

Let $\cG^1$ and $\cG^2$ be two proper standard families. By Theorem~\ref{thm: coupling},
we decompose their total measures as
$\nu_{\cG^i}= \sum_{k=1}^\infty \nu^i_k$, $i=1, 2$, such that
$T^{k}_*\nu^1_k=T^{k}_*\nu^2_k$ and
$\sum_{k>n}\nu^i_k (M)\le (1-\Theta_\bc)^{n/N_\bc}$.
Therefore,
{\allowdisplaybreaks
\beqyn
\left\| T^n_* \nu_{\cG^1} - T^n_*\nu_{\cG^2} \right\|_{TV}
&\le & \left\| \sum_{k=1}^n T_*^{n-k}\left( T^{k}_*\nu^1_k-T^{k}_*\nu^2_k \right)\right\|_{TV} \\
& & +  \left\| T^n_*\sum_{k>n}\nu^1_k \right\|_{TV} +  \left\| T^n_*\sum_{k>n}\nu^2_k \right\|_{TV} \\
&\le & \left\| \sum_{k>n}\nu^1_k \right\|_{TV} +  \left\| \sum_{k>n}\nu^2_k \right\|_{TV} \le
2 (1-\Theta_\bc)^{n/N_\bc}.
\eeqyn
}This completes the proof of Theorem~\ref{thm: equi}.

\subsection{Proof of Theorem~\ref{thm: acip}}

Let $\theta_0$ be given in \eqref{def theta_0}, and set
\beqn\label{def vartheta_bc}
\vartheta_\bc:=\max\{\theta_0, (1-\Theta_\bc)^{1/N_\bc}\}, \ \
\text{and} \ \ C_\bc=\frac{2}{\theta_0(1-\vartheta_\bc)}.
\eeqn
We first show that there is a probability measure $\mu$ on $M$
such that $T^n_*\nu_\cG=\nu_{T^n\cG}$ converges to $\mu$ in the total variation norm
for any standard family $\cG\in \fF$.
By Lemma~\ref{lem: growth proper},
$T^n\cG$ is a proper standard family for any $n\ge n_\bp(\cG)$,
and note that $\theta_0\le \cZ(\cG)\theta_0^{n_\bp(\cG)}\le 1$.
Apply Theorem~\ref{thm: equi} to the proper standard families
$\cG^1=T^{n_\bp(\cG)}\cG$ and $\cG^2=T^{n_\bp(\cG)+1}\cG$, we get
{\allowdisplaybreaks
\beqyn
\left\|T^{n+1}_*\nu_\cG -T^n_*\nu_\cG \right\|_{TV}
&\le &
\begin{cases}
2, \ & \ n<n_\bp(\cG), \\
2 \vartheta_\bc^{n-n_\bp(\cG)}, \ & \ n\ge n_\bp(\cG)
\end{cases}
\\
&\le & 2 \vartheta_\bc^{n-n_\bp(\cG)}
\le 2 \vartheta_\bc^n \theta_0^{-n_\bp(\cG)}
\le 2\theta_0^{-1} \vartheta_\bc^n \cZ(\cG).
\eeqyn
}It follows that $T^n_*\nu_\cG$ is a Cauchy sequence in the total variation norm,
and hence it converges to some probability measure $\mu$, such that
\beqn
\left\|T^n_*\nu_\cG - \mu \right\|_{TV}
\le \sum_{k=n}^\infty 2\theta_0^{-1} \vartheta_\bc^k \cZ(\cG)
= C_\bc \vartheta_\bc^n \cZ(\cG).
\eeqn
Given another standard family $\cG'\in \fF$, and applying Theorem~\ref{thm: equi}
to $\cG^1=T^{n_\bp(\cG)}\cG$ and $\cG^2=T^{n_\bp(\cG')}\cG'$, we get
\beqn
\left\|T^{n}_*\nu_\cG -T^n_*\nu_{\cG'} \right\|_{TV}\le 2 \vartheta_\bc^{n-\max\{n_\bp(\cG), n_\bp(\cG')\}},
\eeqn
for any $n\ge \max\{n_\bp(\cG), n_\bp(\cG')\}$.
Therefore, $T^n_*\nu_\cG'$ converges to the same measure $\mu$.

It is obvious that $\mu$ is $T$-invariant.
It remains to show that $\mu$ is absolutely continuous, that is,
$m(A)>0$ for any Borel subset $A\subset M$ with $\mu(A)>0$.
To see this, we
consider the Lebesgue standard pair $\cG_0=(M, m)$, then there is
a large $n\ge 1$ such that $\|T^n_*m - \mu\|_{TV}\le 0.5 \mu(A)$, and thus
$m(T^{-n}A)\ge 0.5\mu(A)>0$.
Since $T$ is non-singular with respect to $m$, we must have $m(A)>0$.

\section{Proof of Theorem~\ref{thm: density}}

To prove Theorem~\ref{thm: density}, we need the following preparations.

\subsection{Second growth lemma}

We recall an alternative definition of the characteristic $\cZ$
function (see Section 5 in \cite{MR2540156} or \S 7.4 in \cite{MR2229799} with $q_0=1$).
Given an interval $W\subset M$ and a point $x\in W$, we denote
$r_W(x):=\dist(x, \partial W)$, that is, the Euclidean distance from $x$ to the closest endpoint of $W$.
Further,
given a family $\cG=\sum_{\alpha\in \cA} \lambda_\alpha (W_{\alpha}, \nu_\alpha)$
and a point $x\in W_\alpha$, we shall denote $r_{\cG}(x)=r_{W_\alpha}(x)$ if the choice of $\alpha$ is clear.
We then denote
\beqn
\wcZ(\cG)
:=\sup_{\eps>0}  \dfrac{\nu_\cG\left( r_\cG<\eps \right)}{\eps^{q_0}}
=\sup_{\eps>0}
\dfrac{\sum_{\alpha\in \cA} \lambda_\alpha \nu_\alpha \left\{x\in W_\alpha:\ r_{W_\alpha}(x)<\eps \right\}}{\eps^{q_0}}.
\eeqn
Using the fact that $m(r_W<\eps)=\min\{2\eps, |W|\}$ and Lemma~\ref{lem: sp property}, it is easy to show that
$
\wcZ(\cG) \le 2e^{2C_\br} \cZ(\cG)
$
for any standard family $\cG$.

The growth lemma that we establish in Lemma~\ref{lem: growth}
is usually called the \emph{first growth lemma},
which immediately implies the following \emph{second growth lemma}.

\begin{lemma}\label{second growth}
For any $\eps>0$ and
any standard pair $\cG=(W, \nu)$,
we have
\beq\label{second growth lem}
\nu\left(r_{W, n}(x)<\eps\right):=
\nu\left\{ x\in W:\ r_{T^n \cG}(T^n x) <\eps \right\} < C_\bp \eps^{q_0}
\eeq
for all $n> q_0 \log_{\theta_0} |W|$,
where
$q_0$, $\theta_0$ are given in \eqref{def theta_0} and
$C_\bp$ is given in \eqref{choose Zp}.
\end{lemma}

\begin{proof}
By Lemma~\ref{lem: growth} and the choice of constants in \eqref{choose Zp},
for any $\eps>0$,
any standard pair $\cG=(W, \nu)$ and any $n> q_0 \log_{\theta_0} |W|$, we have
\beqn
\wcZ(T^n\cG)\le 2e^{2C_\br} \cZ(T^n\cG)\le 2e^{4C_\br} \left( \theta_0^n/|W|^{q_0} + c_0  \right)
\le 2e^{4C_\br}(1+c_0)< C_\bp.
\eeqn
In other words, \eqref{second growth lem} holds
for any $\eps>0$.
\end{proof}

Lemma~\ref{second growth} is a slight generalization of
the second growth lemma in \S 5.9 of \cite{MR2229799}, in which
$q_0=1$ and $\cG$ is restricted to a normalized Lebesgue standard pair.
To avoid confusion, we point out that we use the notation $m_W(\cdot)$ to represent
the normalized Lebesgue measure on $W$ in this paper, while
$m_W(\cdot)$ is the unnormalized one in \cite{MR2229799}.

\subsection{Hofbauer tower and liftability}\label{sec: Hofbauer}

In order to show that the invariant density $h=d\mu/dm$ is a dynamically H\"older
series, we first need to construct the corresponding collection $\cW_h$ of supporting intervals.
To this end, we introduce a Markov extension over the system
$(M, T, \xi_1)$ which is nowadays called \emph{Hofbauer tower}.
For references on this subject, see
\cite{MR570882, MR599481, MR1026617, MR1328254, MR1469107, MR1714974, MR2322177, MR2408392, MR3249887}, etc.

For our purpose, we construct the Hofbauer tower as follows:
we set $\cD_0:=\{M\}$ and for $n\ge 1$,
\beqn
\cD_n:=\left\{ T\left(W\cap V\right):\ W\in \xi_1 \ \text{and} \ V\in \cD_{n-1} \right\}.
\eeqn
It is not hard to see that $\cD_n=\{T^n W_\alpha: \ \alpha\in M/\xi_n\}$, that is,
$\cD_n$ is the collection of components of $T^nM$.
We further set $\cD=\cup_{n\ge 0} \cD_n$, which is a collection of countably many intervals.
The \emph{Hofbauer tower extension} over $(M, T, \xi_1)$ is the triple $(\hM, \hT, \hxi \ )$ where
\begin{itemize}
\item[(1)]
the tower is given by
$\hM:=\left\{(x, D)\in M\times \cD: \ x\in \oD\right\}$;
\item[(2)]
the map
$\hT: \hM\backslash \pi^{-1}(\cS_1) \to \hM$ is given by
$
\hT(x, D)=(T(x), T(D\cap W(x))),
$
where $W(x)$ is the interval in $\xi_1$ containing $x$ and
$\pi: \hM\to M$ is the canonical projection, i.e., $\pi(x, D)=x$;
\item[(3)]
the partition of $\hM$ is given by $\hxi:=\{\hD \}_{D\in \cD}$,
where for any interval $D\in \cD$,
we set $\hD:=\left\{(x, D):  x\in \oD \right\}$,
which is an identical copy of $\oD$.
\end{itemize}
It is easy to see that $\hxi$ is a Markov partition for $\hT$. Also,
$\hT$ is an extension of $T$ via the projection $\pi$, i.e., $\pi\circ \hT=T\circ \pi$.
By extending the Euclidean metric of the unit interval $M$ to the tower $\hM$ in a natural way,
we have that $\hM$ is a complete separable metric space,
which is not necessarily to be compact unless the map $T$ is already Markov.
For any $D\in \cD$, we define the \emph{level of $D$} as
\beqn
\ell(D):=\min\{n\ge 0:\ D\in \cD_n \}.
\eeqn
Further, for any $\hx=(x, D)\in \hM$, we define the level of $\hx$ as $\ell(\hx)=\ell(D)$.
Then we set the \emph{$n$-level set of $\hM$} to be $\hM_n:=\{\hx\in \hM: \ \ell(\hx)=n\}$.
In particular, we call $\hM_0$ the base of the tower $\hM$, which is an identical copy of $M$.

We now discuss the liftability property of the Lebesgue measure.
Let $\fB$ be the Borel $\sigma$-algebra of $M$, then by extension,
$\hfB:=\hxi\vee \pi^{-1}\fB$ is the Borel $\sigma$-algebra of $\hM$.
We then extend the normalized Lebesgue measure $m$ on $M$ to
a (possibly infinite) measure $\om$ on $\hM$ by setting
$
\om(A)=\sum_{D\in \cD} m\left(\pi\left( A\cap \hD\right)\right)
$
for any $A\in \hfB$. Define a sequence of measures on $\hM$ by
\beq\label{def omn}
\om_n(A)=\om\left(\hT^{-n}A\cap \hM_0\right), \ \ \text{for any} \ n\ge 0.
\eeq
Note that $\om_n$ are all probability measures and
$\pi_* \om_n=T_*^n m$, that is, $\om_n$ projects to $T_*^n m$,
or equivalently, we say that $T_*^n m$ is lifted to $\om_n$.
Similarly, we denote the Cesaro means of $\om_n$ by $\wm_n$, that is,
\beqn
\wm_n:=\frac{1}{n}\sum_{k=0}^{n-1} \om_k, \ \ \text{for any} \ n\ge 1.
\eeqn
Note that $\wm_n$ projects to $\frac{1}{n}\sum_{k=0}^{n-1}  T_*^k m$.
We say that $\{\wm_n\}_{n\ge 0}$ is \emph{liftable} if
$\wm_n$ has a subsequence which converges weak
star to a non-vanishing, in fact, probability measure on $\hM$.
To show the liftability, we will prove that

\begin{lemma}\label{lem: tight}
The sequence of measures $\wm_n$ is tight, i.e.,
for any $\delta>0$, there exists a compact subset $F\subset  \hM$ such that
$\wm_n(\hM\backslash F)<\delta$ for all $n$.
\end{lemma}

\begin{proof}
It suffices to show that $\om_n$ is tight.
Choose $\eps_0>0$ such that $C_\bp \eps_0^{q_0}<\delta/2$,
where $q_0$ and $C_\bp$ are given by
\eqref{def theta_0} and \eqref{choose Zp} respectively.
Since $\xi_1$ is a generating partition,
we can choose $L\in \IN$
such that $\xi_L=\bigvee\limits_{k=0}^{L-1} T^{-k} \xi_1$
has diameter smaller than $\eps_0$.
Furthermore,
we may assume $C_\bc \vartheta_\bc^L<\delta/4$,
where $C_\bc$ and $\vartheta_\bc$ are the constants
given by Theorem~\ref{thm: acip}.
We then set
\beqn
E:=\left\{\hx \in \hM:\ \ell(\hx)\le L \right\}.
\eeqn

By the definition of $\om_n$ in \eqref{def omn},
it is easy to see that when $n\le L$, the measure $\om_n$ is supported on $E$
and thus $\om_n(\hM\backslash E)=0$.
When $n>L$, we consider the Lebesgue standard pair $\cG_0=(M, m)$,
and we denote $r_{M, k}(x):=r_{T^k \cG_0} (T^kx)$ for any $k\ge 0$.
For any $\hx=(x, M)\in \hM_0$,
if $r_{M, n-L}(x)\ge \eps_0$,
i.e., $\dist(T^{n-L} x, \partial D_{n-L})\ge \eps_0$,
where we denote $\hT^{n-L} (\hx)=(T^{n-L} x, D_{n-L})$,
then there is $\alpha\in \xi_L$ such that
$T^{n-L} x\in  W_\alpha$ and $D_{n-L}$ fully contains $W_\alpha$.
It follows that $\hT^n(\hx)\in E$.
By Lemma~\ref{second growth}, and note that $\log_{\theta_0} |M|=0$,
we have
\beqn
\om_n(\hM\backslash E)=m\left(M\backslash \pi\left( \hT^{-n} E \cap \hM_0 \right) \right)
\le m\left(r_{M, n-L}(x)< \eps_0 \right)
\le C_\bp \eps_0^{q_0}<\delta/2.
\eeqn

Now we construct a compact subset $F$ of $E$ as follows.
Note that $E$ can be rewritten as the following disjoint union
$E=\bigcup_{k=0}^{L} E_k$, where each
\beqn
E_k:=\left\{\hx \in \hM:\ \ell(\hx)=k \right\}
\eeqn
consists of countably many intervals.
For each $k\in [0, L]$, we can pick a subset $F_k\subset E_k$
such that $F_k$ is a union of finitely many intervals
and
\beqn
\sigma \left(\pi\left(E_k\backslash F_k \right)\right)<\frac{\delta}{8L},
\ \text{for measures} \ \sigma=\mu, m, T_*m, \dots, T_*^Lm.
\eeqn
Here $\mu$ is the invariant measure that we obtain in Theorem~\ref{thm: acip}.
It is clear that $F=\bigcup_{k=0}^L F_k$ is a compact subset of $E$.
Moreover, by Theorem~\ref{thm: acip},
\beqyn
\om_n(E\backslash F)
\le m\left( \pi\left(\hT^{-n}(E\backslash F)\right)\right)
&=& T^n_*m\left(\pi(E\backslash F)\right) \\
&\le & \sum_{k=0}^L T^n_*m\left(\pi(E_k\backslash F_k)\right) \\
&\le &
\begin{cases}
(L+1)\cdot \frac{\delta}{8L},  & \ \text{if} \ 0\le n\le L, \\
(L+1)\cdot \frac{\delta}{8L} + C_\bc \vartheta_\bc^L,  & \ \text{if} \ n>L.
\end{cases}
\\
&<& \delta/2.
\eeqyn
Therefore, we have
$\om_n(\hM\backslash F)\le \om_n(\hM\backslash E)+\om_n(E\backslash F)<\delta$.
Hence $\om_n$ is tight, so is $\wm_n$.
\end{proof}

Recall that $\mu$ is the invariant measure that we obtain in Theorem~\ref{thm: acip}.
The following is a direct consequence of Lemma~\ref{lem: tight}.

\begin{lemma}\label{lem: projection}
$\wm_n$ has a subsequence converging
weak star to a probability measure $\hmu$ on $\hM$
such that $\pi_*\hmu=\mu$.
\end{lemma}

\begin{proof}
By Helly-Prohorov theorem, Lemma~\ref{lem: tight} implies that
there is an increasing sequence of natural numbers $\{n_j\}_{j\ge 1}$ such that
$\wm_{n_j}$ converges weak star to a probability measure $\hmu$ on $\hM$.
Applying Theorem~\ref{thm: acip} to the Lebesgue standard pair $\cG_0=(M, m)$,
we have that $T_*^n m$ converges to $\mu$ in total variation,
and hence in the weak star topology as well.
Since
$\pi_* \wm_n=\frac{1}{n}\sum_{k=0}^{n-1}  T_*^k m$, we get
\beqn
\pi_*\hmu=\lim_{j\to \infty} \pi_*\wm_{n_j} = \lim_{j\to \infty} \frac{1}{n_j}\sum_{k=0}^{n_j-1}  T_*^k m=\mu.
\eeqn
Here the above limits are taken in the weak star topology.
\end{proof}

\subsection{Pesin-Sinai decomposition}

In this section, we would like to show that the invariant measure $\mu$ on $M$
is the total measure of a standard family.
By Lemma~\ref{lem: projection},
we shall instead show that the lifted measure $\hmu$ on $\hM$
has the following structure.

\begin{definition}\label{def Pesin-Sinai decomposition}
A probability measure $\hnu$ on $\hM$ is said to have \emph{Pesin-Sinai decomposition}
if the conditional decomposition of $\hnu$
with respect to the countable partition $\hxi=\{\hD\}_{D\in \cD}$
has the following form:
\beqn
\hnu(A)=\sum_{D\in \cD} \ \lambda(D) \cdot \hnu_{D}(A)
\eeqn
for any $A\in \hfB$, where
\begin{itemize}
\item[(1)] $\left\{\lambda(D) \right\}_{D\in \cD}$ is a probability vector on $\cD$,
that is,  $0\le \lambda(D) \le 1$ for
any $D\in \cD$ and $\sum_{D\in \cD} \lambda(D)=1$;
\item[(2)] $\hnu_D$ is a probability measure on $\hD$
such that its projection $(D, \pi_*\hnu_D)$ is a standard pair.
\end{itemize}
\end{definition}

\begin{remark}
Definition~\ref{def Pesin-Sinai decomposition} is motivated by the work \cite{MR721733},
in which Pesin and Sinai used a crucial lemma (Lemma 13 therein) to
construct the u-Gibbs measure of partially hyperbolic attractors.
We adapt their notions in our setting.
\end{remark}

If $\hmu$ has Pesin-Sinai decomposition,
then by Lemma~\ref{lem: projection},
it is easy to see that $\mu$ is carried by a standard family.
To this end, we need the following lemma,
which may be regarded as a variant of Lemma 13 of \cite{MR721733}.

\begin{lemma}\label{lem: Pesin-Sinai}
Let $\hnu_n$ be a sequence of probability measures on $\hM$
with the following properties:
\begin{itemize}
\item[(1)]
each $\hnu_n$ has Pesin-Sinai decomposition
$\hnu_n=\sum_{D\in \cD} \lambda_n(D) \cdot \hnu_{n, D}$;
\item[(2)]
let $\rho_{n, D}$ be the density of the standard pair $(D, \pi_*\hnu_{n, D})$,
and assume that $\rho_{n, D}$ converges uniformly in $D$
to a continuous function $\rho_D$ as $n\to \infty$;
\item[(3)] the sequence of measures $\frac{1}{n_j} \sum_{k=0}^{n_j} \hnu_{n_k}$
converges weakly to a measure $\hnu$ on $\hM$, where $n_j$ is a subsequence of natural numbers.
\end{itemize}
Then the measure $\hnu$ has Pesin-Sinai decomposition
$\hnu=\sum_{D\in \cD} \ \lambda(D) \cdot \hnu_{D}$,
such that the density of $(D, \pi_*\hnu_D)$ is exactly given by $\rho_D$.
\end{lemma}

The proof of Lemma~\ref{lem: Pesin-Sinai} is almost
the same as that of Lemma 13 in \cite{MR721733},
by noticing that the uniform limit of regular density is still regular,
as well as that the space of probability vectors
on $\cD$ is weakly compact. Hence we omit the proof here.
In the rest of this subsection, we prove

\begin{lemma}\label{lem: decomposition}
$\hmu$ has Pesin-Sinai decomposition.
\end{lemma}

\begin{proof}
It suffices to show that $\{\om_n\}_{n\ge 0}$
satisfies the first two conditions of Lemma~\ref{lem: Pesin-Sinai},
since the third condition is already shown by Lemma~\ref{lem: projection}.

Recall that $\pi_* \om_n=T_*^n m$ and for any $D\in \cD$,
the interval $\hD$ is an identical copy of $D$ via the projection $\pi$.
Consider the Lebesgue standard pair $\cG_0=(M, m)$, then
$T_*^n m$ is exactly carried by the standard family
\beq\label{def TnG0}
T^n \cG_0 =\sum_{\alpha\in M/\xi_n} m(W_\alpha) \cdot (T^n W_\alpha, \nu_\alpha),
\eeq
where we denote $\xi_n=\{W_\alpha\}_{\alpha\in M/\xi_n}$
and set $\nu_\alpha(\cdot):=m(T^{-n}(\cdot) |W_\alpha)$.
By the construction in Section~\ref{sec: Hofbauer}, it is easy to see that each $T^nW_\alpha\in \cD_n\subset \cD$.
Note that it is possible that $T^nW_\alpha=T^nW_{\alpha'}$ for distinct index $\alpha$ and $\alpha'$.
We would like show that $T^n\cG_0$ is equivalent to a standard family of the form
\beq\label{def TnG01}
T^n \cG_0 \equiv \sum_{D\in \cD} \lambda_n(D) \cdot (D, \nu_{n, D}).
\eeq
To this end, we need to combine standard pairs of \eqref{def TnG0}
over the same interval $D$ as follows.
For simplicity, write $\cA_n:=M/\xi_n$ and for any $D\in \cD$, set
\beqn
\cA_n(D):=\left\{\alpha\in \cA_n:\ T^n W_\alpha=D \right\}.
\eeqn
For any $D\in \cD$ and any $n\in \IN$, if $\cA_n(D)\ne \emptyset$, we define
\beq\label{def lambda nu}
\lambda_n(D)=\sum_{\alpha\in \cA_n(D)} m(W_\alpha) \ \
\text{and} \
\ \nu_{n, D}=\dfrac{\sum_{\alpha\in \cA_n(D)}
m(W_\alpha)\cdot \nu_\alpha}{\sum_{\alpha\in \cA_n(D)} m(W_\alpha)};
\eeq
otherwise, we let $\lambda_n(D)=0$ and $\nu_{n, D}=m_D$.
By Lemma~\ref{lem: mergence}, the pair $(D, \nu_{n, D})$ is a standard pair.
In this way, we obtain the equivalent standard family given by the RHS of \eqref{def TnG01},
whose total measure is $T^n_*m$. By lifting $T^n_*m$ to $\om_n$ and
noting that $\pi^{-1}|_D: D\to \hD$ is trivial, we set $\hnu_{n, D}=(\pi^{-1}|_D)_*\nu_{n, D}$,
then $\om_n$ has Pesin-Sinai decomposition
$\om_n=\sum_{D\in \cD} \lambda_n(D) \cdot \hnu_{n, D}$,
that is, $\{\om_n\}_{n\ge 0}$
satisfies Condition (1) of Lemma~\ref{lem: Pesin-Sinai}.

Before we verify Condition (2) of Lemma~\ref{lem: Pesin-Sinai}, we introduce the following notations.
For any $\alpha\in \cA_n$, we can associated a symbolic representation
$\alpha=(\alpha_{-1}, \alpha_{-2}, \dots, \alpha_{-n})\in (M/\xi_1)^n$ such that
$W_\alpha:=\bigcap_{k=1}^n T^{k-n} W_{\alpha_{-k}}$.
We further denote $\cA_\infty$ the inverse limit space of $\{\cA_n\}_{n\ge 1}$, that is,
if $\alpha=(\alpha_{-1}, \alpha_{-2}, \dots)\in \cA_\infty$, then $W_{\alpha|_n}\ne \emptyset$ for all $n$,
where $\alpha|_n:=(\alpha_{-1}, \alpha_{-2}, \dots, \alpha_{-n})$ is the truncation of first $n$-words of $\alpha$.
For any $D\in \cD$, we set
\beqn
\cA_\infty(D):=\left\{\alpha\in \cA_\infty:\ \alpha|_n\in \cA_n(D) \ \text{for any} \ n \right\}.
\eeqn
Note that $\cA_\infty(D)\ne \emptyset$ for some $D\in \cD$, due to the
Hofbauer tower construction and Assumption \textbf{(H3)}.
We also denote the subset $\cA_n'(D)$ of $\cA_n(D)$ such that
any $\alpha\in \cA_n'(D)$ cannot be extended to an element in $\cA_\infty(D)$.

Now we are ready to check Condition (2) of Lemma~\ref{lem: Pesin-Sinai}.
We shall only show the case given by \eqref{def lambda nu}, since the other case
is trivial. Suppose that $\cA_n(D)\ne \emptyset$.
For any $\alpha=(\alpha_{-1}, \dots, \alpha_{-n})\in \cA_n(D)$,
let $\nu_\alpha$ be the probability measure given by \eqref{def TnG0}.
Similar to \eqref{density formula}, the density of $\nu_\alpha$ is given by
\beqn
\rho_\alpha(x) = \dfrac{1}{m(W_\alpha)} \frac{1}{\left|(T^n)'(x_\alpha)\right|},
\eeqn
where $x_\alpha:=\left(T^n|_{W_\alpha}\right)^{-1}(x)$, i.e.,
$x_\alpha$ is the $n$-th preimage of $x$ in $W_\alpha$.
Alternatively, we define
\beq\label{def p kernel}
p_{n, \alpha}(x, y):=\dfrac{\left|(T^n)'(y_\alpha)\right|}{\left|(T^n)'(x_\alpha)\right|}
=\prod_{k=1}^{n} \frac{\left|T'(y_{\alpha|_{k}})\right|}{\left|T'(x_{\alpha|_{k}})\right|}
\eeq
for all $(x, y)\in D\times D$, and we notice that
\beqn
\dfrac{1}{\left|(T^n)'(x_\alpha)\right|}=\int_D p_{n, \alpha}(x, y) dm(y),
\
m(W_\alpha)=\iint_{D^2} p_{n, \alpha}(x, y) dm(x) dm(y).
\eeqn
Therefore, the measures
$\nu_{n, D}$ given by \eqref{def lambda nu} has density
{\allowdisplaybreaks
\beqyn\label{def rho n D}
\rho_{n, D}(x)
&= &
\dfrac{\sum_{\alpha\in \cA_n(D)} m(W_\alpha)\cdot \rho_\alpha}{\sum_{\alpha\in \cA_n(D)} m(W_\alpha)} \nonumber \\
&=&
\dfrac{\sum\limits_{\alpha\in \cA_\infty(D)}
\int_D p_{n, \alpha|_n}(x, y) dm(y) + \sum\limits_{\alpha\in \cA_n'(D)}  m(W_\alpha)\cdot \rho_\alpha}{\sum\limits_{\alpha\in \cA_\infty(D)} \iint_{D^2} p_{n, \alpha|_n}(x, y) dm(x) dm(y) +\sum\limits_{\alpha\in \cA_n'(D)}  m(W_\alpha) } \label{def rho n D}
\eeqyn
}for any $x\in D$.
On the one hand, for any $\alpha\in \cA_\infty(D)$,
by Assumption \textbf{(H2)} and the formula \eqref{def p kernel},
it is not hard to see that $p_{n, \alpha|_n}$ uniformly converges
to $p_\alpha$ on $D\times D$, where
\beqn
p_{\alpha}(x, y):=\prod_{k=1}^{\infty} \frac{\left|T'(y_{\alpha|_{k}})\right|}{\left|T'(x_{\alpha|_{k}})\right|},
\ \text{for any} \ (x, y)\in D\times D.
\eeqn
On the other hand, we claim that $\sum_{\alpha\in \cA_n'(D)}  m(W_\alpha)\to 0$ as $n\to \infty$.
Indeed,  $\cA_n'(D)=\bigcup_{k>n} \cA_{n, k}'(D)$,
where $\cA_{n, k}'(D)$ consists of all $\alpha\in  \cA_n'(D)$ which cannot be extended to an element
in $\cA_k(D)$.
Note that the sets $\cA_{n, k}'(D)$ is increasing in $k$.
Since $\cG_0=(M, m)$ and all its iterates $T^k_*\cG_0$ are proper standard families,
we apply Theorem~\ref{thm: equi} to $\cG_0$ and $T^k_*\cG_0$ and get
\beqn
\sum_{\alpha\in \cA_{n, k}'(D)}  m(W_\alpha)
\le \left|T^n_*m(D)-T^k_*m(D)\right|
\le 2 (1-\Theta_\bc)^{n/N_\bc}
\eeqn
Since $k$ is arbitrary, we have
$\sum_{\alpha\in \cA_n'(D)}  m(W_\alpha)\to 0$ as $n\to \infty$.
By the above two observations, we conclude that
$\rho_{n, D}$ uniformly converges to
\beqn
\rho_{D}(x):=
\dfrac{\sum_{\alpha\in \cA_\infty(D)} \int_D p_{n, \alpha|_n}(x, y) dm(y) }{\sum_{\alpha\in \cA_\infty(D)} \iint_{D^2} p_{n, \alpha|_n}(x, y) dm(x) dm(y)}.
\eeqn
This completes the verification of Condition (2) of Lemma~\ref{lem: Pesin-Sinai},
and hence $\hmu$ has Pesin-Sinai decomposition.
\end{proof}

\subsection{Proof of Theorem~\ref{thm: density}}

In the previous subsections, we have shown that
the measure $T^n_*m$ is lifted to the measure $\om_n$ given by \eqref{def omn},
which has Pesin-Sinai decomposition.
So its Cesaro mean $\wm_n$ also has Pesin-Sinai decomposition, say,
\beqn
\wm_n=\sum_{D\in \cD} \eta_n(D) \cdot \homega_{n, D}.
\eeqn
For any $\eps>0$, we set
$
\cD_\eps:=\left\{D\in \cD: \ |D|<\eps\right\}.
$
Consider the Lebesgue standard pair  $\cG_0=(M, m)$,
and denote $r_{M, k}(x):=r_{T^k \cG_0} (T^kx)$ for any $k\ge 0$.
By Lemma~\ref{second growth}, we have
\beqn
\sum_{D\in \cD_\eps} \eta_n(D)
\le \frac{1}{n} \sum_{k=0}^{n-1} m\left( r_{M, k}(x) < \eps \right) <C_\bp \eps^{q_0}.
\eeqn

By Lemma~\ref{lem: projection} and Lemma~\ref{lem: decomposition},
there is a subsequence $n_j$ such that $\wm_{n_j}\to \hmu$ in the weak star topology.
Moreover, $\hmu$ has Pesin-Sinai decomposition, say,
\beq\label{hmu P-S}
\hmu=\sum_{D\in \cD} \eta(D) \cdot \hmu_{D}.
\eeq
Moreover, the sequence of probability vectors $\{\eta_{n_j}(D)\}_{D\in\cD}$ converges to
the probability vector $\{\eta(D)\}_{D\in\cD}$ in the weak star topology as $j\to \infty$.
Therefore,
\beq\label{sum eta}
\sum_{D\in \cD_\eps} \eta(D)=\lim_{j\to \infty} \sum_{D\in \cD_\eps} \eta_{n_j}(D)<C_\bp \eps^{q_0}.
\eeq

Now we proceed the proof of Theorem~\ref{thm: density}. Since $\pi_*\hmu=\mu$,
by \eqref{hmu P-S}, the density $h=\frac{d\mu}{dm}$ is given by
\beqn
h=\sum_{D\in \cD} \eta(D) \cdot  \dfrac{d(\pi_*\hmu_D)}{dm}
=:\sum_{D\in \cW_h} h_D,
\eeqn
where we set
$\cW_h:=\left\{ D\in \cD: \ \eta(D)>0\right\}$
and
$h_D=  \eta(D) \cdot  \frac{d(\pi_*\hmu_D)}{dm}$.
Since $(D, \pi_*\hmu_D)$ is a standard pair, by Lemma~\ref{lem: sp property},
we have $\frac{d(\pi_*\hmu_D)}{dm}$ has $L^\infty$-norm bounded by $e^{C_\br}|D|^{-1}$
and dynamically H\"older semi-norm bounded by $C_\br e^{C_\br}|D|^{-1}$.
Hence for any $D\in \cW_h$, we have
\beqn
\left\| h_D \right\|_{D, \bgamma} \le (1+C_\br)e^{C_\br} \eta(D) |D|^{-1}.
\eeqn
For any $s\in (1-\oq, 1]$, as the choice of $q_0$ is flexible and can be arbitrarily close to $\oq$,
it is not harm to assume that $s>1-q_0$. Then we have
\beqn
\left\| h \right\|_{\cW_h, \bgamma, s}
= \sum_{D\in \cW_h} |D|^{s} \left\| h_D \right\|_{D, \bgamma}
\le (1+C_\br)e^{C_\br}  \sum_{D\in \cD} \eta(D) |D|^{s-1}<\infty.
\eeqn
The above convergence is shown as follows: we set $\Gamma_n=\cD_{2^{-n}}\backslash \cD_{2^{-n-1}}$,
by \eqref{sum eta}, we get
{\allowdisplaybreaks
\beqyn
\sum_{D\in \cD} \eta(D) |D|^{s-1}
= \sum_{n=0}^\infty \sum_{D\in \Gamma_n} \eta(D) |D|^{s-1}
&\le& \sum_{n=0}^\infty  2^{(1-s)(n+1)} \sum_{D\in \cD_{2^{-n}}} \eta(D) \\
&\le& \sum_{n=0}^\infty 2^{(1-s)(n+1)}  \cdot C_\bp \left(2^{-n}\right)^{q_0} \\
&=& C_\bp 2^{1-s} \sum_{n=0}^\infty 2^{n(1-q_0-s)}<\infty.
\eeqyn
}This completes the proof of Theorem~\ref{thm: density}.

\section{Proof of Theorem~\ref{thm: mixing}}

We first show that the system is exponential mixing with respect to the Lebesgue measure, that is,

\begin{lemma}\label{lem: Leb mixing}
For any $t\in [0, 1)$, we choose a scale $q_0\le \min\{\oq, 1-t\}$ satisfying \eqref{def theta_0}.
Then for any $f\in \cH_{\cW, \bgamma, t}$ on some collection $\cW$ of countably many intervals
and for any $g\in L^\infty(m)$, we have
\beqn
\left| \int f  g\circ T^n   dm - \int f  dm \int g  d\mu \right|\le
6C_\bc \vartheta_\bc^{n}  \|f\|_{\cW, \bgamma, t} \|g\|_{\infty}.
\eeqn
Here constants $C_\bc$ and $\vartheta_\bc$ are given by Theorem~\ref{thm: acip}.
\end{lemma}

\begin{remark}
Note that the choice of $q_0$ in \eqref{def theta_0} is quite flexible.
It is not hard to see from the proof of Theorem~\ref{thm: acip},
the constants $C_\bc$ and $\vartheta_\bc$ only depend on
the choice of $q_0$, $\delta_0$ and the magnet interval $U$.
As $\delta_0$ and $U$ are fixed but $q_0$ varies, $C_\bc$ and $\vartheta_\bc$ would also vary
depending on the value of $q_0$.
\end{remark}

\begin{proof}[Proof of Lemma~\ref{lem: Leb mixing}]
Without loss of generality, given a function $f\in \cH_{\cW, \bgamma, t}$,
we may assume that $\cW=\{W_\alpha:  \alpha\in \cA\}$
and $f=\sum_{\alpha\in \cA} f_\alpha$ such that
$f_\alpha\not\equiv 0$ on each sub-interval $W_\alpha$.
We define on each $W_\alpha$
two finite measures $\wnu_\alpha^1$ and $\wnu_\alpha^2$
such that their densities are given as follows:
\beqn
\frac{d\wnu_\alpha^1}{dm}=f_\alpha+2K_\alpha \ \
\text{and} \ \ \frac{d\wnu_\alpha^2}{dm}=2K_\alpha,
\eeqn
where $K_\alpha=\|f_\alpha\|_{W_\alpha, \bgamma}>0$.
Note that $\frac{d\wnu_\alpha^1}{dm}\in [K_\alpha, 3K_\alpha]$.
Then we define two families $\cG^i=\sum_{\alpha\in \cA} \lambda_\alpha^i (W_\alpha, \nu_\alpha^i)$,
$i=1, 2$, by
\beqn
\nu_\alpha^i(\cdot)=\wnu_\alpha^i(\cdot\ | W_\alpha), \ \ \text{and} \ \
\lambda_\alpha^i=\frac{\wnu_\alpha^i(W_\alpha)}{\sum_{\alpha\in \cA} \wnu_\alpha^i(W_\alpha)}.
\eeqn
We first show that $\cG^1$ is a standard family in $\fF$. For any $x, y\in W_\alpha$,
{\allowdisplaybreaks
\beqyn
\left| \log \frac{d\nu_\alpha^1}{dm}(x) -  \log \frac{d\nu_\alpha^1}{dm}(y)\right|
&=& \left| \log \frac{f_\alpha(x)+2K_\alpha}{f_\alpha(y)+ 2K_\alpha} \right| \\
&\le & \log \left( 1+ \frac{|f_\alpha(x)-f_\alpha(y)|}{\min\{f_\alpha(x), f_\alpha(y)\}+2K_\alpha} \right) \\
&\le & \frac{|f_\alpha(x)-f_\alpha(y)|}{\min\{f_\alpha(x), f_\alpha(y)\}+2K_\alpha}  \\
&\le & \frac{|f_\alpha|_{W_\alpha, \bgamma} \bgamma^{\bs(x,y)}}{\|f_\alpha\|_{W_\alpha, \bgamma}}
\le C_\br \bgamma^{\bs(x,y)}.
\eeqyn
}Hence each $(W_\alpha, \nu_\alpha^1)$ is a standard pair, and thus $\cG^1$ is a standard family.
Further,  since $\wnu_\alpha^1(W_\alpha)\le 3K_\alpha |W_\alpha|$ and $t\le 1-q_0$, we have
{\allowdisplaybreaks
\beqyn
\cZ(\cG^1)=\sum_{\alpha\in \cA} \lambda_\alpha^1 |W_\alpha|^{-q_0} =
\frac{ \sum_{\alpha\in \cA} \wnu_\alpha^1 (W_\alpha) |W_\alpha|^{-q_0}}{\sum_{\alpha\in \cA} \wnu_\alpha^1(W_\alpha)}
&\le &\frac{3\sum_{\alpha\in \cA} K_\alpha |W_\alpha|^{1-q_0}}{\sum_{\alpha\in \cA} \wnu_\alpha^1(W_\alpha)} \\
&\le &\frac{3 \|f\|_{\cW, \bgamma, t}}{\sum_{\alpha\in \cA} \wnu_\alpha^1 (W_\alpha)} <\infty.
\eeqyn
}Similarly, we can show that $\cG^2$ is a standard family and
\beqn
\cZ(\cG^2)\le \frac{3 \|f\|_{\cW, \bgamma, t}}{\sum_{\alpha\in \cA} \wnu_\alpha^2 (W_\alpha)} <\infty.
\eeqn
By Theorem~\ref{thm: acip},
$
\left\|T^n_*\nu_{\cG^i} - \mu \right\|_{TV}
\le C_\bc \vartheta_\bc^n \cZ(\cG^i),
$
$i=1, 2$, which implies that
\beqn
\left\|T^n_*\left(\sum_{\alpha\in \cA} \wnu_\alpha^i \right)
- \left(\sum_{\alpha\in \cA} \wnu_\alpha^i(W_\alpha) \right) \mu \right\|_{TV}
\le 3 C_\bc \vartheta_\bc^n \|f\|_{\cW, \bgamma, t}.
\eeqn
Therefore, for any $g\in L^\infty$,
{\allowdisplaybreaks
\beqyn
& & \left| \int f g\circ T^n   d m - \int f  d m \int g  d\mu \right| \\
&=& \left|T^n_*\left(\sum_{\alpha\in \cA} \left(\wnu_\alpha^1 -\wnu_\alpha^2 \right)\right)(g)
- \left(\sum_{\alpha\in \cA} \left( \wnu_\alpha^1(W_\alpha) - \wnu_\alpha^2(W_\alpha)\right)\right) \mu(g)\right| \\
&\le & \left\| \left[T^n_*\left(\sum_{\alpha\in \cA} \wnu_\alpha^1 \right)
- \left(\sum_{\alpha\in \cA} \wnu_\alpha^1(W_\alpha) \right) \mu \right] - \right. \\
& & \ \ \left. \left [T^n_*\left(\sum_{\alpha\in \cA} \wnu_\alpha^2 \right)
- \left(\sum_{\alpha\in \cA} \wnu_\alpha^2(W_\alpha) \right) \mu \right] \right\|_{TV} \|g\|_{\infty} \\
&\le & 6C_\bc \vartheta_\bc^{n} \|f\|_{\cW, \bgamma, t} \|g\|_{\infty}.
\eeqyn
}This completes the proof of Lemma~\ref{lem: Leb mixing}.
\end{proof}

Now we are ready to prove Theorem~\ref{thm: mixing}.
For any $t\in [0, \oq)$,
again as the choice of $q_0$ is flexible, we may set
\beqn
q_0:=\frac{\oq-t}{2} \ \text{and} \ s:=1-\frac{\oq+t}{2}.
\eeqn
It is obvious that $s\in (1-\oq, 1]$, then by Theorem~\ref{thm: density}, the
invariant density $h=d\mu/dm\in \cH_{\cW_h, \bgamma, s}$.
We denote the collection $\cW_h=\{V_\beta: \beta\in \cB\}$ and
write $h=\sum_{\beta\in \cB} h_\beta$, where each $h_\beta\in \cH_{V_\beta, \bgamma}$.

For any $f\in \cH_{\cW, \bgamma, t}$ with a collection
$\cW=\{W_\alpha: \alpha\in \cA\}$, we write $f=\sum_{\alpha\in \cA} f_\alpha$,
where each $f_\alpha\in \cH_{W_\alpha, \bgamma}$.
Set the joint collection by
$\cW\vee \cW_h:=\{W_\alpha\cap V_\beta: \ \alpha\in \cA, \ \beta\in \cB\}$.
Then we can write $fh=\sum_{\alpha\in \cA} \sum_{\beta\in \cB} f_\alpha h_\beta$, and
{\allowdisplaybreaks
\beqyn
& & \sum_{\alpha\in \cA} \sum_{\beta\in \cB} |W_\alpha\cap V_\beta|^{t+s}
\|f_\alpha h_\beta \|_{W_\alpha\cap V_\beta, \bgamma} \\
&\le & \sum_{\alpha\in \cA} |W_\alpha|^{t} \|f_\alpha  \|_{W_\alpha, \bgamma}
\cdot \sum_{\beta\in \cB} |V_\beta|^s \|h_\beta \|_{V_\beta, \bgamma}
\le \|f\|_{\cW, \bgamma, t} \|h\|_{\cV, \bgamma, s}.
\eeqyn
}In other words, $fh\in \cH_{\cW\vee \cW_h, \bgamma, t+s}$ such that
$
\|f\|_{\cW\vee \cV, \bgamma, t+s}\le \|f\|_{\cW, \bgamma, t} \|h\|_{\cV, \bgamma, s}.
$
Note that the scale is $q_0=1-(t+s)$, and note that the constants $C_\bc$ and $\vartheta_\bc$
in Lemma~\ref{lem: Leb mixing} depend on $q_0$ and thus on $t$.
By Lemma~\ref{lem: Leb mixing},
{\allowdisplaybreaks
\beqyn
\left| \int f  g\circ T^n   d\mu - \int f  d\mu \int g  d\mu \right|
&=& \left| \int fh  g\circ T^n   dm - \int fh dm \int g  d\mu \right| \\
&\le & 6C_\bc \vartheta_\bc^{n} \|fh\|_{\cW\vee \cW_h, \bgamma, t+s} \|g\|_{\infty} \\
&\le & C_t  \vartheta_t^{n}\|f\|_{\cW, \bgamma, t} \|g\|_{\infty},
\eeqyn
}where $C_t=6C_\bc \|h\|_{\cW_h, \bgamma, s}$ and $\vartheta_t=\vartheta_\bc$.
This finishes the proof of Theorem~\ref{thm: mixing}.

\section{Proof of Theorem~\ref{thm: ASIP}}\label{Sec: Proof of ASIP}

Let $f\in \cH^{ad}_{\cW, \bgamma, t}$ be a function satisfying all the conditions
in Theorem~\ref{thm: ASIP}, and let $\sigma_f^2$ be given by \eqref{sigma_f}.
If $\sigma_f^2=0$, then it is well known that $f$ is a coboundary up to a constant,
i.e., $f=g-g\circ T+\IE(f)$ for some
$g\in L^2(\mu)$ (see e.g. Theorem 18.2.2 in \cite{MR0322926}), and thus the ASIP is automatic.
In the rest of the proof, we concentrate on the case when $\sigma_f^2>0$.

Given an integrable function $f: M\to \IR$ and a measurable partition $\xi$ of $M=[0, 1]$,
we denote by $\IE(f|\xi)$ the conditional expectation of $f$ with respect to $\xi$.
We also denote by $\sigma(\xi)$ the Borel $\sigma$-algebra on $M$ generated by $\xi$.

We recall the following result in \cite{MR0433597} (see also \S 7.9 in \cite{MR2229799}).

\begin{proposition}\label{prop: ASIP}
Suppose there exist constants $\eps\in (0, 2]$ and $C>0$ such that
\begin{itemize}
\item[(1)] $f\in L^{2+\eps}(\mu)$;
\item[(2)] for all $m\ge 1$,
$\left\| f- \IE\left(f|\xi_m\right) \right\|_{L^{2+\eps}(\mu)} \le C m^{-(2+7/\eps)}$;
\item[(3)] Suppose that $\sigma_f^2>0$ and
$\Var\left( \sum_{k=0}^{n-1} f\circ T^k \right)=n\sigma_f^2 + \cO(n^{1-\eps/30})$;
\item[(4)] For any $n\ge 1$ and $m\ge 1$,
$\left| \mu(A\cap B) - \mu(A)\mu(B)\right| \le C n^{-168(1+2/\eps)}$
for any $A\in \sigma(\xi_m)$ and $B\in \sigma(T^{-(n+m)}\xi_\infty)$.
\end{itemize}
Then the stationary process $\{f\circ T^n\}_{n\ge 0}$ satisfies the ASIP.
\end{proposition}

Now we continue to prove Theorem~\ref{thm: ASIP} by verifying conditions
in Proposition~\ref{prop: ASIP} as follows:
\begin{itemize}
\item Since $f\in \cH^{ad}_{\cW, \bgamma, t}\subset L^{1/t}(\mu)$, where $t<\frac12$,
then Condition (1) holds by taking $\eps=\min\left\{2, \frac{1}{t} -2 \right\}$.
\item
To check Condition (2), we denote the adapted collection
$\cW=\{W_\alpha: \alpha\in \cA\}$ of countably many intervals
such that the endpoints of $W_\alpha$ belong to $\cS_{n(\alpha)}$ for some $n(\alpha)\in \IN$.
In other words, $W_\alpha\in \sigma\left(\xi_{n(\alpha)} \right)$.
Then we rewrite $f\in \cH^{ad}_{\cW, \bgamma, t}$ as $f=\sum_{\alpha\in \cA} f_\alpha$,
such that $f_\alpha\in \cH_{W_\alpha, \bgamma}$.
Now for every interval $W\in \xi_m$ and any two points $x, y\in W$, we have $\bs(x, y)\ge m$ and thus
{\allowdisplaybreaks
\beqyn
|f_\alpha(x)-f_\alpha (y)| &\le &
\begin{cases}
2\|f_\alpha\|_\infty, \ & \ \text{if} \ s(x,y)< n(\alpha), \\
|f_\alpha|_{\cH_{W_\alpha, \bgamma}} \bgamma^{\bs(x, y)}, \ & \ \text{if} \ s(x,y)\ge n(\alpha),
\end{cases} \\
&\le & 2\|f_\alpha\|_{\cH_{W_\alpha, \bgamma}} \bgamma^{m-n(\alpha)},
\eeqyn
}which implies that
$\left\| f_\alpha- \IE\left(f_\alpha |\xi_m\right) \right\|_{\infty}\le
2\|f_\alpha\|_{\cH_{W_\alpha, \bgamma}} \bgamma^{m-n(\alpha)}$.
Also, note that if $n(\alpha)\le m$, then both $f_\alpha$ and $\IE\left(f_\alpha |\xi_m\right)$
are supported on $W_\alpha$.

Note that $2+7/\eps=\max\left\{\frac{11}{2}, \ \frac{2+3t}{1-2t}\right\}<a$,
where $a$ is given by \eqref{fast tail}. Set $b=a/(2+7/\eps)$.
By Minkowski's inequality, as well as
\eqref{Holder partition ad} and \eqref{fast tail},
{\allowdisplaybreaks
\beqyn
\left\| f- \IE\left(f|\xi_m\right) \right\|_{L^{2+\eps}(\mu)}
&\le & \sum_{\alpha\in \cA}  \left\| f_\alpha- \IE\left(f_\alpha |\xi_m\right) \right\|_{L^{1/t}(\mu)} \\
&\le & \sum_{\alpha\in \cA: \ n(\alpha)<m^{\frac{1}{b}}}
\left\| f_\alpha- \IE\left(f_\alpha |\xi_m\right) \right\|_{\infty} \mu(W_\alpha)^{t}  \\
& & + 2 \sum_{\alpha\in \cA: \ n(\alpha)\ge m^{\frac{1}{b}}} \left\| f_\alpha \right\|_{L^{1/t}(\mu)} \\
&\le & 2\|f\|^{ad}_{\cW, \gamma, t} \bgamma^{m-m^{\frac{1}{b}}}+
\cO\left( m^{-a/b} \right) \\
&=& \cO\left(m^{-(2+7/\eps)} \right).
\eeqyn
}Hence Condition (2) holds.
\item
Note that the series in \eqref{sigma_f} converges absolutely
by Condition \eqref{cor 1m}.
By direct computation, we have
{\allowdisplaybreaks
\begin{eqnarray*}
\Var\left( \sum_{k=0}^{n-1} f \circ T^k\right)
&= & n \sigma_f^2 -  2\sum_{k>n} n  \Cov\left(f, f\circ T^k\right)  -2 \sum_{k=1}^{n-1} k \Cov\left(f, f\circ T^k\right) \\
&=& n  \sigma_f^2 +\cO\left(n^{1-\frac{1}{15}} \right)= n \sigma_f^2 +\cO\left(n^{1-\eps/30} \right).
\end{eqnarray*}
}Therefore, Condition (3) holds.
\item
By the $T$-invariance of $\mu$, it suffices to show Condition (4) with $m=1$.
Note that any $A\in \sigma(\xi_1)$ is a disjoint union of intervals in $\xi_1$.
We take $f_A=\bbone_A+1$, then $f_A\in \cH_{M, \bgamma}=\cH_{\{M\}, \bgamma, 0}$
such that
\beqn
\|f_A\|_{M, \bgamma} =
\|f_A\|_\infty + |f_A|_{\cH_{M, \bgamma}}
\le 2+1/\bgamma.
\eeqn
Also, $B\in \sigma(T^{-(n+1)}\xi_\infty)$ means that there is a Borel measurable subset $B'\subset M$
such that $B=T^{n+1} B'$, and thus $\bbone_B=\bbone_{B'}\circ T^{n+1}$.
Therefore, by Theorem~\ref{thm: mixing},
{\allowdisplaybreaks
\beqyn
\left|\mu(A\cap B)-\mu(A)\mu(B) \right| &=&
\left| \int f_A \cdot \bbone_{B'}\circ T^{n+1} d\mu - \int f_A d\mu \int \bbone_{B'} d\mu \right| \\
&\le & C_0 \vartheta_0^{n+1}   \| f_A \|_{M, \bgamma} \| \bbone_{B'} \|_{\infty} \\
&=& C_0 (2+1/\bgamma) \vartheta_0^{n+1},
\eeqyn
}which indicates Condition (4).
\end{itemize}
To sum up, any function $f\in \cH^{ad}_{\cW, \bgamma, t}$
satisfying all the conditions in Theorem~\ref{thm: ASIP}
also satisfies the four conditions in Proposition~\ref{prop: ASIP},
and hence the stationary process $\{f\circ T^n\}_{n\ge 0}$ satisfies the ASIP.
The proof of Theorem~\ref{thm: ASIP} is complete.

\section{Examples and Remarks}\label{sec: rem app}
We shall revisit several examples which were previously studied in the literature.
Applying our results to these examples, we could reinterpret some known results and
make some generalizations.

\subsection{A class of piecewise linear maps}\label{sec: rem q-scale}

In this subsection, we  consider a class of piecewise linear map on $M=[0, 1]$
with infinitely many inverse branches. More precisely,
given a sequence of positive numbers $\{a_k\}_{k\ge 1}$ such that $\sum_{k\ge 1} a_k=1$.
Set $b_0=1$ and for $k\ge 1$,
\beqn
b_k=1-\sum_{m=1}^{k} a_m=\sum_{m=k+1}^\infty a_m.
\eeqn
It is clear that $\xi_1:=\{W_k\}_{k\ge 1}$ is a partition of $M=[0, 1]$,
where $W_k:=(b_k, b_{k-1}]$.
Pick another sequence $\{\Lambda_k\}_{k\ge 1}$ of positive numbers
such that $\Lambda_k\ge 2$. Moreover, we assume that
$a_k\Lambda_k\ge b_1$ for $k=1, 2$, and
$a_k\Lambda_k\ge b_{k-2}$ for any $k> 2$.
Then we define a piecewise linear map $T: M\to M$ by setting
\begin{equation}\label{equ:piecewiselinear}
T(x)=
\begin{cases}
0, \ & \ x=0, \\
\Lambda_k(x-b_k), \ & \ x\in W_k.
\end{cases}
\end{equation}

Albeit $T$ is piecewise linear, the existence of acip (absolutely continuous $T$-invariant probability measure)
heavily depends on the above parameters. We emphasize that even if the map is Markov,
the ``big image property'' (i.e., $\inf_{k\ge 1} \left|TW_k\right|>0$) does not hold, and hence the classical
theory of Gibbs-Markov systems is not applicable in our situation.
We recall some results for these piecewise linear maps in earlier literature.
\begin{itemize}
\item[(1)] Rychlik \cite{MR728198} showed that if $\sum_{k\ge 1} \Lambda_k^{-1} <\infty$, then
$T$ admits an acip which enjoys the exponential mixing. Rychlik also constructed a counter-example,
that is, $T$ does not admit an acip if $a_k=2^{-k}$ and $\Lambda_k=2$.
\item[(2)] Bruin and Todd studied in \cite{MR2959300} a class of piecewise linear maps\footnote{By personal communication, Bruin and Todd named such map as the vSSV map, because it was introduced by van Strien to Stratmann and Vogt. This map has a bearing on the existence and nature of wild attractors in interval dynamics, see \cite{MR1370759}.}, which is a simplified
linear model of the induced map of the Fibonacci unimodal map.
To be precise, given any $\lambda\in(0,1)$, we set $a_{k}:=\lambda^{k-1}(1-\lambda)$ for all $k\ge 1$
and thus $b_{k}:=\lambda^{k}$ for all $k\ge 0$.
Meanwhile, put $\Lambda_{1}:=1/a_{1}$ and $\Lambda_{k}:=1/a_{2}$ for all $k\geq 2$.
The corresponding map is denoted by $T_\lambda$.
Bruin and Todd showed that $T_{\lambda}$ admits an acip if and only if $\lambda\in(0, \frac12)$.
Moreover, whenever $\lambda\in (0,\frac12)$,
they also showed that the invariant density restricting on each $W_k$ is a constant equals to
\beq\label{B-T density}
\frac{v_k}{|W_k|}:= \dfrac{\frac{1-2\lambda}{\lambda}\cdot(\frac{\lambda}{1-\lambda})^{k}}{a_k}
=\dfrac{(1-\lambda)(1-2\lambda)}{(1-\lambda)^k}.
\eeq
\end{itemize}

The following proposition provides a sufficient condition for the existence of acip,
when $T$ is the piecewise linear map given by \eqref{equ:piecewiselinear}.

\begin{proposition}\label{prop:application}
If there is $q\in(0,1]$ such that
\beq\label{suff expansion}
\inf_{k\ge 2} \Lambda_k> 2, \ \text{and} \
\limsup_{N\to \infty} \frac{\sum_{k=N+1}^\infty a_k^{1-q} \Lambda_k^{-q}}{b_N^{1-q}} < 1,
\eeq
then the piecewise linear map $T$ admits an acip,
which satisfies the exponential decay of correlation and almost sure invariant principle.
\end{proposition}

\begin{remark}
It is not hard to see that Rychlik's condition $\sum_{k\ge 1} \Lambda_k^{-1} <\infty$ is stronger than
Condition \eqref{suff expansion}. Therefore, the results in \cite{MR656227, MR728198} are recovered by our coupling method.
Also, Condition \eqref{suff expansion} never holds for any $q\in(0,1]$ if $\Lambda_k=2$ for all $k$, which
corresponds to the absence of acip.
\end{remark}

\begin{proof}[Proof of Proposition \ref{prop:application}]
It is obvious that Assumption \textbf{(H2)} holds
since the log Jacobian $\log|T'|$ is constant on each interval $W_k\in \xi_1$.

We next verify that $T$ satisfies Assumption \textbf{(H3)} by showing the second branch $W_2$ is a magnet interval.
By our definition, it is easy to see $TW_1\supset W_2$, $TW_2\supset W_2$,
and $TW_k\supset \cup_{m=k-1}^\infty W_m\supset W_{k-1}$ for any $k>2$.
Hence a component of $T^n W_k$ must contain $W_2$ for any $n\ge k-2$.
For any interval $W\subset M$, by the uniform expansion with rate $\Lambda_k>2$,
$T^{n_0} W$ must be cut by $\cS_1=\{b_k\}_{k\ge 1}$ for some positive integer $n_0\le -\log_2 |W|$.
We pick a component $V$ of $T^{n_0} W$ whose left endpoint belongs to $\cS_1$,
then $TV\supset W_\ell$ for some $\ell>2$. Therefore, at least one component
of $T^n W$ contains $W_2$ for any $n\ge n_W:=n_0+\ell-1$,
which implies that $W_2$ is a magnet.

Finally, we focus on the validity of Assumption \textbf{(H1)}.
Indeed, let $W$ be an interval of length less than a sufficiently small $\delta>0$.
\begin{itemize}
\item
If $W$ is away from the accumulation point $0$,
then it only intersects two consecutive intervals in $\xi_1$, say $W_k$ and $W_{k+1}$,
and thus
\beqn
\sum_{{\alpha\in W/\xi_1}}
 \left(\frac{|W|}{|TW_{\alpha}|}\right)^{{q}}
 \frac{|W_{\alpha}|}{|W|}\le \left(\frac{1}{\Lambda_k} + \frac{1}{\Lambda_{k+1}} \right)^{q}
 \le \left(\frac{1}{2} + \frac{1}{\inf_{k\ge 2} \Lambda_k} \right)^{q}.
\eeqn
\item
Otherwise, if $W$ is close to $0$,
without loss of generality, we may assume $W=[0, b_N]=\cup_{k=N+1}^\infty W_k$
for sufficiently large $N$.
Then
\beqn
\sum_{{\alpha\in W/\xi_1}} \left(\frac{|W|}{|TW_{\alpha}|}\right)^{{q}} \frac{|W_{\alpha}|}{|W|}
=\sum_{k=N+1}^\infty \left(\frac{|W|}{|TW_k|}\right)^{{q}} \frac{|W_k|}{|W|}
=\frac{\sum_{k=N+1}^\infty a_k^{1-q} \Lambda_k^{-q}}{b_N^{1-q}}
\eeqn
\end{itemize}
In other words, Condition \eqref{suff expansion} guarantees
Assumption  \textbf{(H1)} - the Chernov's one-step expansion holds at $q$-scale in either of the above cases.

Applying Theorems~\ref{thm: acip},~\ref{thm: density},~\ref{thm: mixing} and~\ref{thm: ASIP}, we can
deduce all the assertions of Proposition \ref{prop:application}.
\end{proof}

We now provide two particular examples of piecewise linear maps
which satisfy Condition~\eqref{suff expansion} and thus Proposition~\ref{prop:application}.
\begin{itemize}
\item[(1)] In spirit of Rychlik's results and counter-example in \cite{MR728198}, we consider the
piecewise linear map with  $a_k=2^{-k}$ and $\Lambda_k=k$. It is straightforward that $b_N=2^{-N}$ and
for any $q\in (0, 1)$, we have
\beqn
\frac{\sum_{k=N+1}^\infty a_k^{1-q} \Lambda_k^{-q}}{b_N^{1-q}} =
\frac{\sum_{k=N+1}^\infty 2^{k(q-1)} k^{-q} }{2^{N(q-1)}}\le \frac{2^{q-1} N^{-q}}{1-2^{q-1}} \to 0
\eeqn
as $N\to \infty$, and hence Condition~\eqref{suff expansion} holds.
\item[(2)] Let $T_\lambda$ be the piecewise linear map that Bruin and Todd studied in \cite{MR2959300}.
Given any $\lambda\in(0,1)$, we recall that
$a_{k}:=\lambda^{k-1}(1-\lambda)$ for all $k\ge 1$ and thus $b_{k}:=\lambda^{k}$ for all $k\ge 0$.
Moreover, $\Lambda_{1}:=1/a_{1}$ and $\Lambda_{k}:=1/a_{2}$ for all $k\geq 2$.
We claim that
$T_\lambda$ satisfies Condition \eqref{suff expansion} and thus Proposition \ref{prop:application}
if and only if $\lambda\in(0, \frac12)$, which agrees with the results of Bruin and Todd in \cite{MR2959300}.
Indeed, it is easy to see that for any $\lambda\in (0, 1)$,
$$
\inf_{k\geq 2} \Lambda_{k}=\frac{1}{\lambda(1-\lambda)}\geq 4.
$$
Meanwhile,
{\allowdisplaybreaks
\beqy\label{B-T}
\frac{\sum_{k=N+1}^\infty a_k^{1-q} \Lambda_k^{-q}}{b_N^{1-q}}
&=&\frac{\sum_{k=N+1}^{\infty}[\lambda^{k-1}(1-\lambda)]^{1-q}\cdot[\frac{1}{\lambda(1-\lambda)}]^{-q}}{\lambda^{(1-q)N}} \nonumber\\
&=& \frac{\lambda^{q}(1-\lambda)}{1-\lambda^{1-q}}.
\eeqy
}
It is not hard to check that \eqref{B-T} is less than $1$ if and only if $\lambda^{1-q}<1-\lambda$,
and hence \eqref{B-T} is less than $1$ for some $q\in (0, 1)$ if and only if  $\lambda\in(0,\frac12)$.
In other words, Condition~\eqref{suff expansion} holds if and only if $\lambda\in(0,\frac12)$.

We remark that when $\lambda\in(0, \frac12)$, the invariant density given by \eqref{B-T density}
is a dynamically H\"older series, which agrees with our Theorem \ref{thm: density}.
More precisely, it is straightforward to check that the invariant density belongs to $\cH_{\cW, \bgamma, s}$, where
$\cW=\{W_k\}_{k\ge 1}$, for any $\bgamma\in (0, 1)$, and for any $s\in (0, 1)$ such that $\lambda^s<1-\lambda$.

\end{itemize}


\subsection{Certain unbounded observables}\label{sec: unbounded}

Let  $T: M=[0, 1]\to M$ be a one-dimensional map satisfying
Assumption \textbf{(H1)}, i.e., the one-step expansion at $q$-scale,
and recall that $\oq$ is the supremum of such $q$ given in \eqref{def oq}.
It directly from \textbf{(H1)} that $T$ is uniformly expanding, i.e.,
there exists $\Lambda>1$ such that
$\inf_{x\in M\backslash \cS_\infty} |T'(x)|>\Lambda$.
It is easy to see that the separation time $\bs(\cdot, \cdot)$
in Definition~\ref{def: sep} induces a weaker metric on $M$, that is,
there exists $C>0$ such that
\beqn
|x-y|\le C\Lambda^{-\bs(x, y)}, \ \ \text{for any}\ x, y\in M.
\eeqn
Let $\bgamma$ be the constant given by \eqref{regular constant}, which can be taken
arbitrarily close to $1$, and set $\kappa:=-\log_{\Lambda} \bgamma>0$.
If $f$ is a $\kappa$-H\"older function on an interval $W\subset M$,
then $f$ is also a dynamically H\"older function on $W$ with parameter $\bgamma$
such that
$|f|_{W, \bgamma}\le C^{-\kappa} |f|_{C^\kappa(W)}$.
Note that in applications, we could always take $\gamma$ arbitrarily close to $1$.

As pointed out in Remark~\ref{rem: observable q}, the space
$\cH_{\cW, \bgamma, t}$ with $t>0$ would contain some unbounded observables.
For instance, for any $\tau\in (0, \oq)$,
we consider the unbounded function
\beq\label{def f unbounded}
f(x)=
\begin{cases}
x^{-\tau}, \ & \ 0<x\le 1, \\
0, \ & \ x=0.
\end{cases}
\eeq

\begin{remark}
This function was studied in Gou\"ezel's note \cite{Gou},
in which $T$ is the doubling map. He showed that $f$ satisfies a stable law when $\tau\ge \frac12$, and
he also pointed out that $f$ satisfies a CLT when $\tau\in (0, \frac12)$, using the criteria by Dedecker \cite{MR2099550}.
We shall show below that the ASIP holds in the latter case.
\end{remark}

It is clear that $f\in \cH_{\cW, \bgamma, t}$ for any $t\in (\tau, \oq)$
and some $\bgamma$ close to $1$,
where the collection is chosen to be $\cW=\{W_k:=(2^{-k}, 2^{-k+1}]\}_{k\ge 1}$.
Indeed, we set $\kappa:=-\log_{\Lambda} \bgamma$, and
write $f=\sum_{k\ge 1} f_k$ with $f_k=f\bbone_{W_k}$, then
{\allowdisplaybreaks
\beqyn
\|f_k\|_{W_k, \bgamma} & \lesssim & \|f_k\|_{L^\infty(W_k)} + |f_k|_{C^\kappa(W_k)} \\
& \lesssim & \|f_k\|_{L^\infty(W_k)} + |f_k'|_{L^\infty(W_k)}  |W_k|^{1-\kappa} \\
& \lesssim & 2^{k\tau} + 2^{k(\tau+1)} 2^{-k(1-\kappa)} \lesssim 2^{k(\tau+\kappa)}.
\eeqyn
}Thus, if we choose $\bgamma$ close to $1$ such that $\kappa:=-\log_{\Lambda} \bgamma<t-\tau$, then
\beq\label{f tau Holder}
\|f\|_{\cH_{\cW, \bgamma, t}}=\sum_{k\ge 1} |W_k|^{t} \|f_k\|_{W_k, \bgamma}
\lesssim \sum_{k\ge 1} 2^{k(\tau+\kappa-t)} < \infty.
\eeq
By Theorem~\ref{thm: mixing},
the correlations between
any unbounded function $f\in \cH_{\cW, \bgamma, t}$ and any bounded observable $g\in L^\infty(m)$ decays
exponentially fast. \\

Finally, we discuss the space $\cH^{ad}_{\cW, \bgamma, t}$ with fast tail,
for which the ASIP applies by Theorem~\ref{thm: ASIP}.
For simplicity, we consider the doubling map $T: x\mapsto 2x \pmod 1$,
with the partition $\xi_1=\{[0, \frac12], (\frac12, 1]\}$ and invariant measure $\mu=m$.
We claim that if $\tau\in (0, \frac12)$, then
the unbounded function $f$ given by \eqref{def f unbounded} satisfies the ASIP.
Indeed,
\begin{itemize}
\item
The collection $\cW=\{W_k:=(2^{-k}, 2^{-k+1}]\}_{k\ge 1}$ is adapted such that $n(k)=k$.
Pick any $t\in \left(\tau, \frac12 \right)$, it follows from \eqref{f tau Holder} that $f\in \cH^{ad}_{\cW, \bgamma, t}$.
\item Moreover, $f$ has fast tail since
\beqn
\sum_{k\ge  n} \|f_k\|_{L^{1/t}(\mu)}\le \sum_{k\ge  n} \|f_k\|_\infty |W_k|^t
=\sum_{k\ge  n} 2^{k(\tau-t)}
=\cO\left( 2^{n(\tau-t)} \right).
\eeqn
\item The auto-correlations condition \eqref{cor 1m} holds since
the Fourier coefficients of $f$ satisfy that
$
a_k:=\int_0^1 x^{-\tau} e^{i2\pi kx} dx \lesssim k^{\tau-1},
$
and thus
\beqn
\left| \Cov(f, \ f\circ T^n) \right|=\sum_{k=1}^\infty a_k a_{k\cdot 2^n}
\lesssim 2^{n(\tau-1)} \sum_{k\ge 1} k^{2(\tau-1)} = \cO\left( \left(2^{\tau-1}\right)^n \right).
\eeqn
\end{itemize}
Therefore, the unbounded function $f$ given by \eqref{def f unbounded} satisfies the ASIP and thus the CLT.\\

\medskip

\appendix

\section{Proof of Lemma~ \ref{lem: split} and~ \ref{lem: covering}  }\label{sec: appendix}

\subsection{Proof of Lemma \ref{lem: split}}



\begin{proof} By the formula of $\cZ(\cdot)$ and \eqref{G split 1}, for any standard family $\cG$,
{\allowdisplaybreaks
\begin{eqnarray*}
\left|\dfrac{ \cZ(\hcG)}{\cZ(\cG)} -1 \right|
&\le &\dfrac{\sum_{\alpha\in \bcA} \frac{\brho (1-\bdelta)}{1-\brho\bdelta} \lambda_\alpha |W_\alpha|^{-q_0}
+\sum_{\alpha\in \cA\backslash \bcA} \frac{\brho \bdelta}{1-\brho\bdelta}\lambda_\alpha |W_\alpha|^{-q_0}}
{\sum_{\alpha\in \bcA} \lambda_\alpha |W_\alpha|^{-q_0}
+\sum_{\alpha\in \cA\backslash \bcA} \lambda_\alpha |W_\alpha|^{-q_0}} \\
&\le & \max\left\{ \frac{\brho (1-\bdelta)}{1-\brho\bdelta},  \  \frac{\brho \bdelta}{1-\brho\bdelta} \right\}
\le \dfrac{\brho}{1-\brho}.
\end{eqnarray*}
}Moreover, for any $\alpha\in \bcA$, we have $W_\alpha=U$, and the density of
$\dfrac{\nu_{\alpha}-\brho\ m_{W_\alpha}}{1-\brho}$ is given by
$\dfrac{\rho_{\alpha}-\brho}{1-\brho}$. By Lemma~\ref{lem: sp property},
for any $x, y\in W_\alpha$,
{\allowdisplaybreaks
\beqyn
& & \left|\log \dfrac{\rho_{\alpha}(x)-\brho}{1-\brho} - \log \dfrac{\rho_{\alpha}(y)-\brho}{1-\brho}\right| \\
&\le & \left|\log \rho_\alpha(x) -\log \rho_\alpha(y)\right|
+\log\left(1+ \brho \dfrac{|\rho_\alpha(x)^{-1}-\rho_\alpha(y)^{-1}|}{1-\brho \max\{\rho_\alpha(x)^{-1}, \rho_\alpha(y)^{-1}\}} \right) \\
&\le & C_\br \bgamma^{\bs(x,y)} + \frac{\brho C_\br e^{C_\br}}{(1- \brho e^{C_\br})|U|} \bgamma^{\bs(x,y)}.
\eeqyn
}We use the fact $\log(1+z)\le z$ for any $z\ge 0$ in the last inequality.
Hence, for any $\alpha\in \bcA$,
\beqn
\left|\log \left( \frac{\rho_\alpha - \brho}{1-\brho}\right)\right|_{W_\alpha, \bgamma}
\le C_\br + \frac{\brho C_\br e^{C_\br}}{(1- \brho e^{C_\br})|U|}.
\eeqn
Therefore, we can choose $\brho_\bc$ small enough such that  for any $\brho \in (0, \brho_\bc)$ and
any standard family $\cG$, we have that
$\cZ(\hcG)\le \cZ(\cG)/\theta_0$, and
the density of each pair in $\hcG$ satisfies \eqref{non-standard 1} and \eqref{non-standard 2}.
By Remarks~\ref{remark non-standard 1} and~\ref{remark non-standard 2}, we have that
$T\hcG$ is a standard family, and
\beqn
\cZ(T\hcG)\le e^{4C_\br} \left( \cZ(\hcG)\theta_0 + c_0 \right)\le e^{4C_\br} \left(\cZ(\cG) + c_0 \right).
\eeqn
The proof of this lemma is complete.
\end{proof}

\subsection{Proof of Lemma \ref{lem: covering}}
We first choose an integer $k\ge 1$ such that
$(k/3)^{q_0} \ge 2C_\bp$, where $C_\bp$
is the proper constant that we choose in \eqref{choose Zp}.
We then divide  $M=[0,1]$
into $k$ sub-intervals $W_1, W_2, \dots, W_{k}$ of equal length.
For each Lebesgue standard pair
$\cG_i=(W_i, m_{W_i})$, by Assumption (\textbf{H3}), there exists $n_{W_i}\ge 1$
such that for any $n\ge n_{W_i}$,
at least one component of $T^n(W_i)$ contains $U$, which means that $\delta(T^n\cG_i)>0$.
We set
\beqn
n_\bc:=\max\{n_\bp, \max_{1\le i\le k} n_{W_i}\}, \ \ \text{and}\ \ d_\bc':=\min_{1\le i\le k} \delta(T^{n_\bc}\cG_i).
\eeqn

For any proper standard family $\cG=\sum_{\alpha\in \cA} \lambda_\alpha  (W_{\alpha}, \nu_\alpha)$,
we denote
$
\cA_0:=\left\{\alpha\in \cA: \ |W_\alpha|\ge 3/k \right\},
$
then
\beqn
\sum_{\alpha\in \cA_0} \lambda_\alpha = 1-\sum_{\alpha\not\in \cA_0} \lambda_\alpha
\ge 1 - \left(\frac{k}{3}\right)^{-q_0} \cZ(\cG) \ge \frac12.
\eeqn
For any $\alpha\in \cA_0$, there exists $1\le i_\alpha\le N$ such that $W_\alpha$ contains $W_{i_\alpha}$.
We then cut the Lebegue standard pair $(W_\alpha, m_{W_\alpha})$ by the two endpoints of $W_{i_\alpha}$,
and obtain a new standard family $\cG_\alpha'$. Note that the weight of $W_{i_\alpha}$ in $\cG_\alpha'$ is
$\frac{|W_{i_{\alpha}}|}{|W_\alpha|}\ge \frac{1}{k}$.
By \eqref{delta p1}, \eqref{delta p2} and \eqref{delta p3}, we have
{\allowdisplaybreaks
\begin{eqnarray*}
\delta(T^{n_\bc}\cG)\ge \sum_{\alpha\in \cA_0}  \lambda_{\alpha} \delta(T^{n_\bc}(W_\alpha, \nu_\alpha))
&\ge &\sum_{\alpha\in \cA_0} \lambda_{\alpha} e^{-C_\br}\delta(T^{n_\bc}(W_\alpha, m_{W_\alpha})) \\
&\ge  & \sum_{\alpha\in \cA_0} \lambda_{\alpha} e^{-C_\br}\delta(T^{n_\bc}(\cG_\alpha')) \\
&\ge & \sum_{\alpha\in \cA_0} \lambda_{\alpha} e^{-C_\br} \frac{1}{k} \delta(T^{n_\bc}(W_{i_\alpha}, m_{W_{i_\alpha}})) \\
&\ge & \sum_{\alpha\in \cA_0} \lambda_{\alpha} e^{-C_\br} \frac{1}{k} d_\bc' \\
&\ge & \frac{e^{-C_\br} d_\bc' }{2k} =: d_\bc.
\end{eqnarray*}
}This completes the proof of the lemma.

\bigskip

\noindent\textbf{Acknowledgement}

J. Chen would like to thank Tianyuan Mathematical Center in Southwest China (TMCSC),
where part of the work was done, for their hospitality and support (Grant 11826102).

Y.-W. Zhang is partially support by the NSFC grant 11701200 and 11871262.
Y.-W. Zhang would also like to thank AMS China exchange program Ky and Yu-Fen Fan fund travel grant
for visiting University of Massachusetts Amherst, where part of this work was carried out.

\bibliography{1dcoupling2020_updated}

\bibliographystyle{abbrv}

\end{document}